\documentclass[12pt,reqno]{amsart}
\usepackage{amsmath,amsthm,amsfonts,amssymb,amscd,amstext}
\usepackage[utf8]{inputenc}
\usepackage[dvips]{graphicx}
\usepackage{euscript}
\usepackage{a4wide}
\usepackage{mathtools}

\numberwithin{equation}{section}

\usepackage[colorlinks=true,linkcolor=red,urlcolor=black]{hyperref}

\newcommand{\qand}{\quad\text{and}\quad}

\theoremstyle{plain}
\newtheorem{maintheorem}{Theorem}

\newtheorem{theorem}{Theorem}[section]

\newtheorem{corollary}[theorem]{Corollary}
\newtheorem{lemma}[theorem]{Lemma}

\theoremstyle{definition}
\newtheorem{remark}[theorem]{Remark}
\newtheorem{definition}{Definition}
\newtheorem{example}{Example}
\newtheorem{conjecture}{Conjecture}

\renewcommand{\angle}{\sphericalangle}

\newcommand{\RR}{{\mathbb R}}

\newcommand{\ZZ}{{\mathbb Z}}

\newcommand{\sS}{{\mathbb S}}

\newcommand{\la}{\lambda}

\renewcommand{\epsilon}{\varepsilon}

\newcommand{\sing}{\mathrm{Sing}}
\newcommand{\vol}{\operatorname{vol}}
\newcommand{\mm}{\operatorname{m}}

\newcommand{\V}{\EuScript{V}}
\newcommand{\U}{\EuScript{U}}

\newcommand{\J}{\EuScript{J}}

\newcommand{\wt}{\widetilde}
\newcommand{\wh}{\widehat}

\begin{document}

\title
{Adapted metrics for singular hyperbolic flows}



\renewcommand{\subjclassname}{\textup{2000} Mathematics5
  Subject Classification}


\author{Vitor Araujo}

\thanks{V.A. is partially supported by CNPq-Brazil (grant
  301392/2015-3) and FAPESB-Brazil (grand PIE0034/2016).
  L.S. is partially supported Fapesb-JCB0053/2013, PRODOC-UFBA/2014 and CNPq. V.C. is supported by CAPES}

\author{Vinicius Coelho}

\author{Luciana Salgado}

\address[V.A.]{Universidade Federal da Bahia,
Instituto de Matem\'atica e Estat\'\i stica\\
Av. Adhemar de Barros, S/N , Ondina,
40170-110 - Salvador-BA-Brazil}
\email{vitor.d.araujo@ufba.br, vitor.araujo.im.ufba@gmail.com}

\address[V.C.]{Universidade Federal do Oeste da Bahia, Centro Multidisciplinar de Bom Jesus da Lapa\\
Av. Manoel Novais, 1064, Centro, 47600-000 - Bom Jesus da Lapa-BA-Brazil}
\email{viniciuscs@ufob.edu.br}

\address[L.S.]{Universidade Federal do Rio de Janeiro, Instituto de
   Matem\'atica\\
   Avenida Athos da Silveira Ramos 149 Cidade Universit\'aria, P.O. Box 68530, 
   21941-909 Rio de Janeiro-RJ-Brazil }
 \email{lsalgado@im.ufjr.br, lucianasalgado@ufrj.br}

\begin{abstract}
  Singular and sectional-hyperbolic sets are the objects of the
  extension of the classical Smale Hyperbolic Theory to flows having
  invariant sets with singularities accumulated by regular orbits
  within the set. It is by now well-known that (partially) hyperbolic
  sets admit adapted metrics.

  We show the existence of singular-adapted metrics for any
  singular-hyperbolic set with respect to a $C^{1}$ vector field on
  finite dimensional compact manifolds. Moreover, we obtain
  sectional-adapted metrics for certain open classes of
  sectional-hyperbolic sets and also for any hyperbolic set.
\end{abstract}


\date{\today}

\maketitle
\tableofcontents

\section{Introduction}

Let $M$ be a connected compact finite $m$-dimensional
manifold, $m \geq 3$, with or without boundary. We consider
a vector field $X$, such that $X$ is inwardly transverse to
the boundary $\partial M$, if $\partial M\neq\emptyset$. The
flow generated by $X$ is denoted by $X_t$.

A hyperbolic set for a flow $X_t$ on a finite dimensional Riemannian
manifold $M$ is a compact invariant set $\Gamma$ with a continuous
splitting of the tangent bundle,
$T_\Gamma M= E^s\oplus E^X \oplus E^u$, where $E^X$ is the direction
of the vector field, for which the subbundles are invariant under the
derivative $DX_t$ of the flow $X_t$
\begin{align}\label{eq:hyp-splitting}
  DX_t\cdot E^*_x=E^*_{X_t(x)},\quad  x\in \Gamma, \quad t\in\RR,\quad *=s,X,u;
\end{align}
and $E^s$ is uniformly $(K,\lambda)$-contracted by $DX_t$ and $E^u$ is
likewise $(K,\lambda)$-expanded: there are $K,\lambda>0$ so that
\begin{align}\label{eq:Klambda-hyp}
  \|DX_t\mid_{E^s_x}\|\le K e^{-\lambda t},
  \quad
  \|(DX_t \mid_{E^u_x})^{-1}\|\le K e^{-\lambda t},
  \quad x\in \Gamma, \quad t>0.
\end{align}
Very strong properties can be deduced from the existence of
such hyperbolic structure; see for
instance~\cite{Bo75,BR75,Sh87,KH95,robinson2004}.

An important feature of a hyperbolic structure is that it
does not depend on the metric on the ambient manifold; see
\cite{HPS77}. We recall that a metric is said to be
\emph{adapted} to the hyperbolic structure if we can take
$K = 1$ in equation \eqref{eq:Klambda-hyp}.

Weaker notions of hyperbolicity (e.g. dominated splitting, partial
hyperbolicity, volume hyperbolicity, sectional-hyperbolicity,
singular-hyperbolicity) have been developed to encompass larger
classes of systems beyond the uniformly hyperbolic ones;
see~\cite{BDV2004} and
specifically~\cite{viana2000i,AraPac2010,ArbSal} for singular
hyperbolicity and Lorenz-like attractors.

In the same work \cite{HPS77}, Hirsch, Pugh and Shub asked about
adapted metrics for dominated splittings. This was given a positive
answer by Gourmelon \cite{Goum07} in 2007, where adapted metrics are
obtained for dominated splittings of both diffeomorphisms and flows,
and also for partially hyperbolic splittings.


Proving the existence of some hyperbolic structure is, in
general, a non-trivial matter, even in its weaker forms.  In
\cite{lewow80}, Lewowicz proved that a diffeomorphism on a
compact riemannian manifold is Anosov if and only if its
derivative admits a nondegenerate Lyapunov quadratic
function.  An example of application of the adapted metric
from \cite{Goum07} is contained in \cite{arsal2012a} where,
following the spirit of Lewowicz's result, the authors
construct quadratic forms which characterize partially
hyperbolic and singular-hyperbolic structures on a trapping
region for flows.

In \cite{arsal2015a}, the existence of an adapted metric for
singular-hyperbolic splittings was proved for
three-dimensional vector fields through the use of quadratic
forms.  In \cite[Theorem B]{salvinc2017}, the existence of
adapted metrics for any singular hyperbolic splitting having
a one-dimensional uniformly contracting bundle was obtained,
extending the result from \cite{arsal2015a} for any
codimension-one singular-hyperbolic set. This was done under
the point of view of $\J$-algebras of Potapov (see e.g.
\cite{Wojtk01}) and through the use of quadratic forms.

Here, we extend this result for any singular-hyperbolic set and also
for open classes of sectional-hyperbolic sets, using only multilinear
algebra and the dynamics of the tangent map to the flow.

\section{Preliminary definitions and results} \label{sec:stat}

We now present preliminary definitions and results. Recall
that a \emph{trapping region} $U$ for a flow $X_t$ is an
open subset of the manifold $M$ which satisfies: $X_t(U)$ is
contained in $U$ for all $t>0$, and there exists $T>0$ such
that $\overline{X_t(U)} $ is contained in the interior of
$U$ for all $t>T$. We define
$\Gamma(U)=\Gamma_X(U):= \cap_{t>0}\overline {X_t(U)}$ to be
the \emph{maximal positive invariant subset in the trapping
  region $U$}.

A \emph{singularity} for the vector field $X$ is a point
$\sigma\in M$ such that $X(\sigma)=\vec0$ or, equivalently,
$X_t(\sigma)=\sigma$ for all $t \in \RR$. The set formed by
singularities is the \emph{singular set of $X$} denoted
$\sing(X)$.  We say that a singularity is hyperbolic if the
eigenvalues of the derivative $DX(\sigma)$ of the vector
field at the singularity $\sigma$ have nonzero real part.

\begin{definition}\label{def:domination}
  A \emph{dominated splitting} over a compact invariant set $\Lambda$ of $X$
  is a continuous $DX_t$-invariant splitting $T_{\Lambda}M =
  E \oplus F$ with $E_x \neq \{0\}$, $F_x \neq \{0\}$ for
  every $x \in \Lambda$ and such that there are positive
  constants $K, \lambda$ satisfying
  \begin{align}\label{eq:def-dom-split}
    \|DX_t|_{E_x}\|\cdot\|DX_{-t}|_{F_{X_t(x)}}\|<Ke^{-\la
      t}, \ \textrm{for all} \ x \in \Lambda, \ \textrm{and
      all} \,\,t> 0.
  \end{align}
\end{definition}

A compact invariant set $\Lambda$ is said to be
\emph{partially hyperbolic} if it exhibits a dominated
splitting $T_{\Lambda}M = E \oplus F$ such that subbundle
$E$ is \emph{uniformly contracted}, i.e., there exists $C>0$
and $\lambda>0$ such that
$\|DX_t|_{E_x}\|\leq Ce^{-\lambda t}$ for $t\geq 0$.  In
this case $F$ is the \emph{central subbundle} of $\Lambda$.
Or else, we may replace uniform contraction along $E$ by
\emph{uniform expansion} along $F$, i.e., the right hand
side condition in \eqref{eq:Klambda-hyp}.

The following is an extension of the notion of $2$-sectional
expansion proposed in \cite[Definition 2.7]{MeMor08}.
We say that a $DX_t$-invariant subbundle
$F \subset T_{\Lambda}M$ is $p$\emph{-sectionally expanding}
if $\dim F_x \geq p$ is constant for $x\in\Lambda$
and there are positive constants $C , \lambda$ such that for
every $x \in \Lambda$ and every $p$-dimensional linear
subspace $L_x \subset F_x$ one has
    \begin{align}\label{eq:def-sec-exp}
      \vert \det (DX_t \vert_{L_x})\vert > C e^{\la t},
      \textrm{ for all } t>0.
    \end{align}

    \begin{definition}\label{def:sechypset}
      \cite[Definition 3]{Salgado19} A $p$-\emph{sectional-hyperbolic
        set} is a partially hyperbolic set whose central subbundle is
      $p$-sectionally expanding.
\end{definition}

This is a particular case of the so called
\emph{singular-hyperbolicity} whose definition we recall now.  A
$DX_t$-invariant subbundle $F \subset T_{\Lambda}M$ is said to be
\emph{volume expanding} if we replace \eqref{eq:def-sec-exp} by
 \begin{align}\label{eq:def-vol-exp}
      \vert \det (DX_t \vert_{F_x})\vert > C e^{\la t},
      \textrm{ for all } t>0.
    \end{align}

\begin{definition}\label{def:singhypset} \cite[Definition
  1]{MPP99} A \emph{singular-hyperbolic set} is a partially hyperbolic
  set whose central subbundle is volume expanding.
\end{definition}

Clearly, in the three-dimensional case, these notions are equivalent.
This is a feature of the Lorenz attractor as proved in \cite{Tu99} and
this notion extends the uniform hyperbolicity for singular flows,
because sectional-hyperbolic sets without singularities are
hyperbolic; see \cite{MPP04, AraPac2010}.

\subsection{Statements of main
  results} \label{sec:statement-result}

We recall the definition of adapted metrics in the
singular-hyperbolic setting, which extends the notion of
adapted metric for dominated and partially hyperbolic
splittings; see e.g.~\cite{Goum07}.

\begin{definition} \label{def:adapsingmetric} We say a
  Riemannian metric $\langle\cdot,\cdot\rangle$ is
  \emph{adapted to a partially hyperbolic splitting} if it
  induces a norm $|\cdot|$ admitting a constant $\lambda>0$
  simultaneously satisfying
    \begin{align*}
      |DX_t\mid_{E_x}|\cdot\big|(DX_t\mid_{F_x})^{-1}|
      \le e^{-\lambda t};
      \quad\text{and}\quad
      |DX_t\mid_{E_x}|\le e^{-\lambda t}
    \end{align*}
    for all $x\in \Gamma$ and $t>0$.
    Additionally, we say that $\langle\cdot,\cdot\rangle$ is
  \begin{enumerate}
  \item \emph{adapted to a singular-hyperbolic splitting} if
   $|\cdot|$ also satisfies
    \begin{align*}
      | \det (DX_t \mid_{F_x})| \ge
      e^{\lambda t} \end{align*}
    for all
    $x\in \Gamma$ and $t>0$. For simplicity we say that
    $\langle\cdot,\cdot\rangle$ is a \emph{singular-adapted
      metric}.
  \item \emph{adapted to a sectional-hyperbolic splitting}
    if the induced norm also satisfies
        \begin{align*}
          | \det (DX_t \mid_{L_x})| \ge
          e^{\lambda t} \end{align*}
        for all
        $x\in \Gamma$, $t>0$ and $2$-subspace $L_x\subset
        F_x$.
        For simplicity we say that
        $\langle\cdot,\cdot\rangle$ is a \emph{sectional-adapted
          metric}.
  \end{enumerate}
\end{definition}

Consider a partially hyperbolic splitting
$T_\Gamma M=E\oplus F$ where $E$ is uniformly contracted and
$F$ is volume expanding. We show that for $C^1$ flows having
a singular-hyperbolic set $\Gamma$ there exists a metric
adapted to the partial hyperbolicity and the area expansion,
as follows.

\begin{maintheorem} \label{mthm:singadaptmetric} Let
  $\Gamma$ be a singular-hyperbolic set for a $C^1$ vector
  field.  Then $\Gamma$ admits a singular adapted metric.
\end{maintheorem}

In the particular case of a uniformly hyperbolic set
$\Gamma$ for a flow, with splitting
$T_\Gamma M=E^s\oplus E^X\oplus E^u = E\oplus F$ with
$E=E^s$ and $F=E^X\oplus E^u$ is sectional-hyperbolic and
admits an adapted metric.

\begin{maintheorem} \label{mthm:sectadaptmetricunif} Let
  $\Gamma$ be a uniformly hyperbolic set for a $C^1$ vector
  field.  Then $\Gamma$ admits a sectional-adapted
  metric.
\end{maintheorem}

Refining the arguments we obtain the following open class of
sectional-hyperbolic attracting sets admitting
sectional-adapted metrics. We say that a vector subbundle
with an invariant splitting $E^c\oplus E^u$ over a compact
invariant subset $\Gamma$ of a $C^1$ vector field is
\emph{strongly partially hyperbolic} if there are
$K>0, \mu>\lambda>0$ so that, for each $t>0$ and $x\in\Gamma$
\begin{align*}
 \|DX_t\mid E^c_x\|\le Ke^{\lambda t}\qand \|DX_{-t}\mid
  E^u_{X_tx}\|<Ke^{-\mu t}.
\end{align*}
Note that this is clearly stronger than the domination condition.

\begin{maintheorem} \label{mthm:sectadaptmetricsing} Let
  $\Gamma$ be a partially hyperbolic set for a $C^1$ vector
  field whose splitting $T_\Gamma M=E^{s}\oplus E^c\oplus E^u$
  is dominated with subbundles of constant dimension and
  such that $E^{s}$ is uniformly contracted; $E^c$ is
  $2$-dimensional and area expanding; $E^u$ is uniformly
  expanded; and $E^c\oplus E^u$ is strongly partially
  hyperbolic.  Then $\Gamma$ admits a sectional-adapted
  metric.
\end{maintheorem}

\subsection{Applications and Conjectures}
\label{sec:statement-result-applic}

\begin{example}
  \label{ex:wild}
  In~\cite{ST98} Turaev and Shilnikov construct a \emph{wild
    attractor} $\Gamma$ (see Figure \ref{fig:wild}) of a vector field
  having a singular hyperbolic splitting $T_\Gamma M=E\oplus F$ where
  $\dim F=3$ and $\dim E=n-3$ for any $n\ge4$. Moreover, the bundle
  $F$ is not sectionally expanding; see \cite[p 296, formula
  (13)]{ST98}. These attractors are also robust: these properties
  persist for all $C^1$ small perturbations of the given vector field.

  Hence Theorem~\ref{mthm:singadaptmetric} ensures that the
  wild attractors constructed in~\cite{ST98} admit a
  singular-hyperbolic metric.
\begin{figure}[htpb!]
    \centering
    \includegraphics[width=3cm]{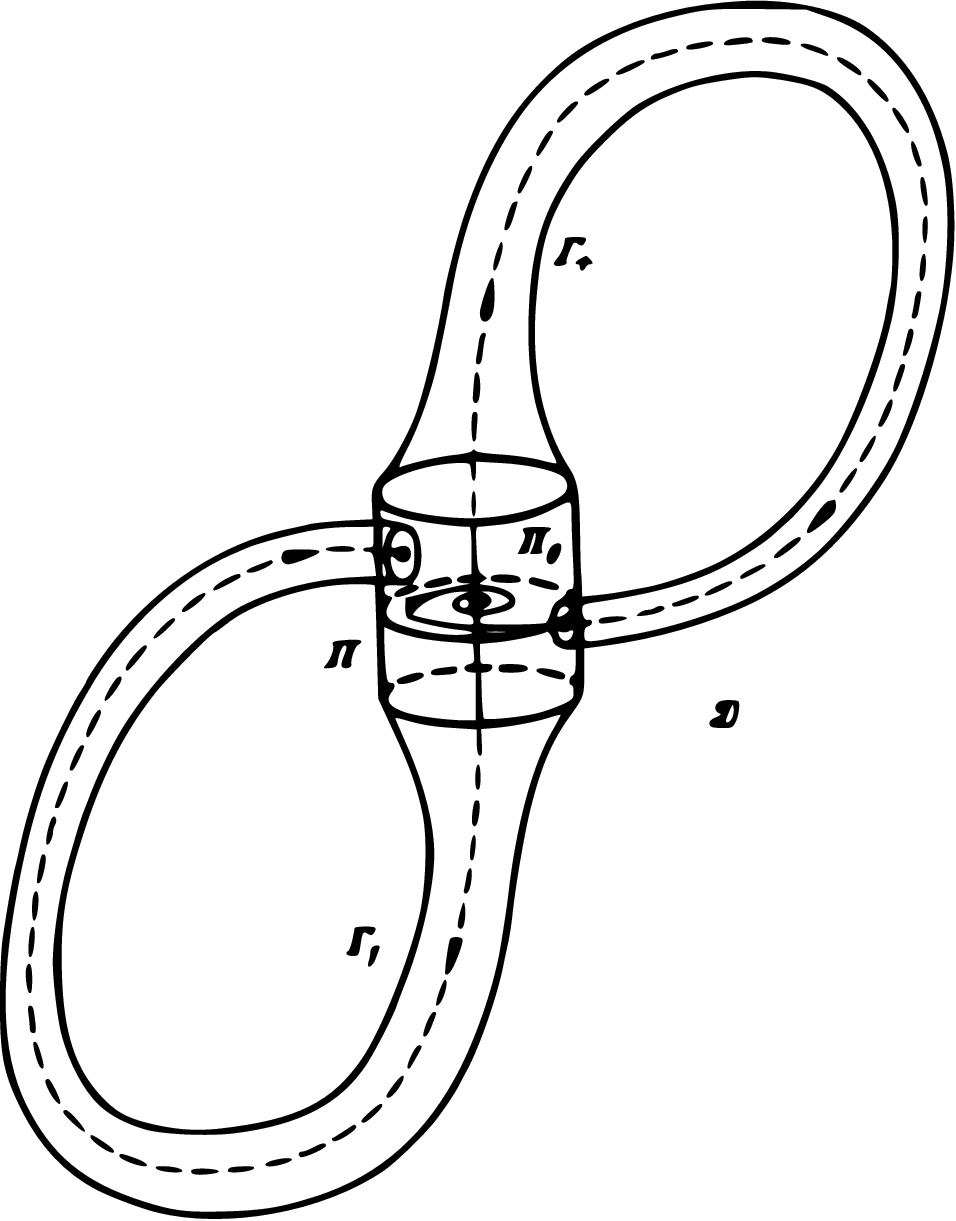}
    \caption{Example of Turaev-Shil'nikov's wild strange attractor.}
    \label{fig:wild}
  \end{figure}
\end{example}

For the next application, we recall that the Riemannian
manifold $M$ is naturally endowed with a volume form $m$
called Lebesgue measure.

\begin{corollary}
  \label{cor:conservative}
  Let $X$ be a $C^1$ volume preserving flow on $M$. Then
  every partially hyperbolic set is singular-hyperbolic and
  thus admits a singular-adapted metric.
\end{corollary}

\begin{proof}
  It is enough to explain that partial hyperbolicity of a
  compact invariant subset $\Gamma$ for a
  conservative vector field $X$ implies
  singular-hyperbolicity, as follows.

  Let $T_\Gamma M=E^{s}\oplus F$ be the singular hyperbolic
  splitting over $\Gamma$. Domination and compactness
  imply that the minimum angle
  \begin{align*}
    \angle(E^{s},F)=\inf\{\arccos \langle u,v\rangle: u\in E_{x}^{s}, v\in
    F_x, \|u\|=\|v\|=1, x\in \Gamma\}
  \end{align*}
  between the bundles is away from zero; see \cite[Appendix
  B]{BDV2004}.  Recall that $E^{s}$ is uniformly contracted.
  Hence, $|\det DX_t|\equiv1$ since $X$ is conservative and
  for all $x\in\Gamma, t>0$
  \begin{align*}
    1&=|\det DX_t(x)|=|\det DX_t\mid E_{x}^{s}|\cdot|\det
    DX_t\mid F_x|\cdot\sin\angle(E_{x}^{s},F_x)
  \end{align*}
so that
  \begin{align*}
    |\det DX_t\mid F_x|
    =
    \big(|\det DX_t\mid E_{x}^{s}|
    \cdot
    \sin\angle(E_{x}^{s},F_x)\big)^{-1}
    \ge (Ke^{-\mu t}\sin\angle(E^{s},F))^{-1}
  \end{align*}
  is uniformly expanding.
\end{proof}

Our methods do not provide that the singular-adapted metric is still
preserved by the flow.

\begin{conjecture}
  \label{conj:conservative}
  Given a singular hyperbolic set for a conservative $C^1$ vector
  field, there exists a singular-adapted metric whose induced volume
  form is still preserved by the flow.
\end{conjecture}

\begin{example}
  \label{ex:sectionalstrong}
  An open class of examples in the setting of
  Theorem~\ref{mthm:sectadaptmetricsing} can be provided by modifying
  the classical Lorenz equations as follows.
  \begin{align*}
    \dot X = a(Y - X);
    \quad
    \dot Y = rX -Y -XZ;
    \quad
    \dot Z = XY - bZ;
    \quad
    \dot W = rW;
  \end{align*}
  where $(X,Y,Z)\in\RR^3$, $W\in\RR^k$ for some fixed $k>1$ and also
  $a=10, b=8/3$ and $r=28$.

  Let $G$ be the vector field on $\RR^3$ provided by the first
  three equations above on $(X,Y,Z)$ and $H$ the vector field on
  $\RR^k$ given by the last equation on $W$. Then the vector field on
  $\RR^{k+3}$ given by the full set of equations can be written as the
  direct product $J=G\times H$. We write $\phi_t$ for the flow of $G$
  and $\psi_t$ for the flow of $J$.

  Let $U$ be the open trapping region of the Lorenz attractor
  $\Lambda$ in $\RR^3$, so that $\Lambda=\cap_{t\in\RR}
  \phi_t(U)$. Then the compact invariant subset
  $\tilde\Lambda=\Lambda\times\{0^k\}=\cap_{t\in\RR}\psi_t(U\times V)$
  of the vector field $J$, where $V$ is any open neighborhood of
  $0^k$ in $\RR^k$, admits a sectional-hyperbolic splitting
  $E^s\oplus F$ with:
  \begin{itemize}
  \item $E^s$ corresponding to the one-dimensional stable (uniformly
    contracting) subbundle of $T_\Lambda\RR^3$; and
  \item $F$ admits a strongly partially hyperbolic splitting
    $F=E^c\oplus E^u$, where
    \begin{itemize}
    \item $E^c$ corresponds to the two-dimensional sectional-expanding
      subbundle of $T_\Lambda\RR^3$; and
    \item $E^u$ is the uniformly expanded $k$-dimensional subbundle
      $0^3\times\RR^k$.
    \end{itemize}
  \end{itemize}
  For the properties of the subbundles $E^s, E^c$ the reader may
  consult \cite{AraPac2010} and references therein.

  This example is clearly robust on the $C^1$ topology of vector
  fields: there exists a neighborhood $\U$ of $W$ so that each
  $Y\in\U$ admits a maximal invariant subset
  $\Lambda_Y=\cap_{t\in\RR}Y_t(U\times V)$ exhibiting a splitting
  $T_{\Lambda_Y}=E_{\Lambda_Y}\oplus F_{\Lambda_Y}$ with the same
  properties as above.

  Hence, Theorem~\ref{mthm:sectadaptmetricsing} ensures that the
  splitting of $T_{\Lambda_Y}$ admits an adapted sectional-hyperbolic
  metric for each $Y\in\U$.
\end{example}

It is now natural to conjecture that the conclusion of
Theorem~\ref{mthm:sectadaptmetricsing} holds in a much more general
setting.

\begin{conjecture}
  \label{conj:sectionalexp}
  Given a $p$-sectional-hyperbolic set $\Gamma$ for a $C^1$
  vector field $X$,  there exists a metric such
  that for some constant $\mu>0$ and all $t>0$
  \begin{itemize}
  \item $|DX_t \mid_{E_{x}}|\le e^{-\mu t}$;
  \item $|DX_t \mid_{E_{x}}| \cdot \big|(DX_t\mid_{F_x})^{-1}| \le e^{-\mu t} $; and
  \item $|\wedge^p DX_t(x) \vert_{L_x}| > e^{\mu t}$ for
    every $p$-dimensional linear subspace
    $L_x \subset F_x$.
  \end{itemize}
\end{conjecture}



\subsection{Organization of the text}
\label{sec:organization-text}

In the present section we provided preliminary definitions in order to
present the statements of the main results together with some
applications and conjectures. In
Section~\ref{sec:statem-prelim-result} we state some auxiliary
results, definitions and prove some useful properties of exterior
products.  In Section~\ref{sec:proof} we give the proofs of our
theorems, divided into three
subsections~\ref{sec:existence-singul-ada},
\ref{sec:uniformly-hyperb-set} and~\ref{sec:open-class-sectional}, one
for each of the Main Theorems~\ref{mthm:singadaptmetric},
\ref{mthm:sectadaptmetricunif} and~\ref{mthm:sectadaptmetricsing},
respectively.

\subsection*{Acknowledgements}
\label{sec:acknowledgements}

This work is part of the PhD. Thesis of V. Coelho developed at the
Mathematics Department of the Federal University of Bahia (UFBA) under
the supervision of L. Salgado. The authors would like to thank the
facilities provided by the Mathematics Institute for the PhD. Program
and the partial financial support from several federal and state
agencies to the Faculty and Students of this Program.


\section{Auxiliary results}
\label{sec:statem-prelim-result}

\subsection{Linear multiplicative cocycles over flows} \label{sec:dominat-linear-multi}

Let $A:G\times\RR\to G$ be a smooth map given by a
collection of linear bijections
\begin{align*}
  A_t(x): G_x\to G_{X_t(x)}, \quad x\in\Gamma, t\in\RR,
\end{align*}
where $\Gamma$ is the base space of the finite dimensional
vector bundle $G$, satisfying the cocycle
property
\begin{align*} A_0(x)=Id, \quad
  A_{t+s}(x)=A_t(X_s(x))\circ A_s(x), \quad x\in\Gamma,
  t,s\in\RR,
\end{align*}
with $\{X_t\}_{t\in\RR}$ a complete smooth flow over
$M\supset\Gamma$.  We note that for each fixed $t>0$ the map
$A_t: G\to G, v_x\in G_x \mapsto A_t(x)\cdot v_x\in
G_{X_t(x)}$ is an automorphism of the vector bundle $G$.

The natural example of a linear multiplicative cocycle over
a smooth flow $X_t$ on a manifold is the derivative cocycle
$A_t(x)=DX_t(x)$ on the tangent bundle $G=TM$ of a finite
dimensional compact manifold $M$. Another example is given
by the exterior power $A_t(x)=\wedge^kDX_t$ of $DX_t$ acting
on $\wedge^k TM$, the family of all $k$-vectors on the
tangent spaces of $M$, for some fixed $1\le k\le\dim G$.

It is well-known that the exterior power of an inner product space has a naturally induced inner product and thus a norm. Thus $G=\wedge^k TM$ has an induced norm from the Riemannian metric of $M$. For more details see e.g. \cite{arnold-l-1998}.

In what follows we assume that the vector bundle $G$ has a
smoothly defined inner product in each fiber $G_x$ which
induces a corresponding norm $\|\cdot\|_x, x\in\Gamma$.

\begin{definition}\label{def:domcocycle}
  A continuous splitting $G=E\oplus F$ of the vector bundle
  $G$  into a pair of subbundles is \emph{dominated} (with
  respect to the automorphism $A$ over $\Gamma$) if
  \begin{itemize}
  \item the splitting is \emph{invariant}: $A_t(x)\cdot
    E_x=E_{X_t(x)}$ and $A_t(x)\cdot F_x=F_{X_t(x)}$ for all
      $x\in \Gamma$ and $t\in\RR$; and
    \item there are positive
  constants $K, \lambda$ satisfying
  \begin{align}\label{eq:def-dom-split-cocycle}
    \|A_t|_{E_x}\|\cdot\|A_{-t}|_{F_{X_t(x)}}\|<Ke^{-\la
      t}, \ \textrm{for all} \ x \in \Gamma, \ \textrm{and
      all} \,\,t> 0.
  \end{align}
  \end{itemize}
\end{definition}

We say that the splitting $G=E\oplus F$ is \emph{partially
  hyperbolic} if it is dominated and the subbundle $E$ is
uniformly contracted: $\|A_t\mid E_x\|\le Ce^{-\mu t}$ for
all $t>0$ and suitable constants $C,\mu>0$.

\subsection{Exterior powers}\label{sec:ext-pow}

Let $V$ be a vector space of dimension $N$.  The $k$th exterior power
of $V$, denoted $\wedge^k(V)$, is the vector space spanned by
alternating (exterior) products of the form
\begin{align*}
  x_{1}\wedge x_{2}\wedge \cdots \wedge x_{k},
  \quad x_{i}\in V, i=1,2,\ldots ,k.
\end{align*}
Every $w\in\wedge^k(V)$ is said to be a $k$-vector. If $w$ can be
expressed as an exterior product of $k$ elements of $V$, then $w$ is
\emph{decomposable}. Decomposable $k$-vectors span $\wedge^k(V)$ but
not every element of $\wedge^k(V)$ is decomposable. \footnote{For
  example, in $\RR^4$ the following $2$-vector (a symplectic form) is
  not decomposable: $e_{1}\wedge e_{2}+e_{3}\wedge e_{4}.$ } Since the
exterior product is alternating, then $\wedge^{N+r}V$ is the trivial
vector space for each integer $r>0$.

\begin{remark}\label{remark3.1}
\begin{enumerate}
\item The dimension of space $\wedge^{r} V$ is
  $\dim \wedge^{r} V = \binom{N}{r}$ for each $1\le r\le N$.  If
  $\{e_{1},\cdots,e_{N}\}$ is a basis of $V$, so the set
  $\{e_{k_{1}}\wedge \cdots \wedge e_{k_{r}}: 1 \leq k_{1} < \cdots <
  k_{r} \leq N\}$ is a basis in $\wedge^{r} V$ with $\binom{N}{r}$
  elements.
\item Each $r$-vector $v_1\wedge\cdots\wedge v_r$ can be
  represented by
  $v_1\wedge \tilde{v_2}\wedge\cdots\wedge\tilde{v_r}$,
  where $v_1,\tilde{v_2}.\dots, \tilde{v_r}$ are mutually
  orthogonal and each $\tilde{v_i}$ belongs to the span of
  $\{v_1,\dots,v_i\}$, for each $i=2,\dots,r$ (just apply
  the Gram-Schmidt ortogonalization procedure).
\item If $V$ admits  an inner product $\langle \cdot , \cdot \rangle$, then
  the bilinear extension of
    $$
    \langle u_{1} \wedge \cdots \wedge u_{r} , v_{1} \wedge
    \cdots \wedge v_{r} \rangle :=\det(\langle u_{i},v_{j}
    \rangle)_{r \times r}
    $$
    defines an inner product in $\wedge^{r}V$. In particular,
    $\|u_{1} \wedge \cdots \wedge u_{r}\| = \sqrt{\det(\langle
      u_{i},u_{j} \rangle)_{r \times r}}$ is a norm on $\wedge^rV$ and
    is also the volume of the $r$-dimensional parallelepiped $H$
    spanned by $u_{1},\cdots,u_{r}$:
    $\vol(u_{1},\cdots,u_{r}) = \vol(H) = \det(H)= |\det
    (u_{1},\cdots,u_{r})|$.
  \item If $A: V \to V$ is a linear operator
    then the linear extension of
    $\wedge^{r} A(u_{1}\wedge\cdots\wedge u_{r}) = A(u_{1})
    \wedge \cdots\wedge A(u_{r})$ defines a linear operator
    $\wedge^{r}A$ on $\wedge^{r}V$.

  \item Let $A: V \to V$ be a linear operator,
    $\wedge^{r} A: \wedge^{r}V \to \wedge^{r}V$ and $G$ a subspace of
    $V$ spanned by $v_{1},\cdots ,v_{s} \in V$. Define $H :=
    A|_{G}$. Then $H(G)$ is spanned by $A(v_{1}),\cdots,A(v_{s})$ and
    $|\det A|_{G}| = \vol (A|_{G}) = \vol (H) = \vol
    (A(v_{1}),\cdots,A(v_{s})) = \|A(v_{1})\wedge \cdots \wedge
    A(v_{s})\| = \|\wedge^{s}A(v_{1}\wedge \cdots \wedge v_{s})\|$.
  \end{enumerate}
\end{remark}

\subsubsection{Exterior power of linear multiplicative cocycles}
\label{sec:exteri-power-linear}

The algebraic construction of exterior power of a vector space can be
applied to any given fiber of a vector bundle and thus obtain the
exterior power of the vector bundle, whose fibers are the
corresponding exterior powers of the original fibers.

It is natural to consider the linear multiplicative cocyle
$\wedge^{k}DX_t$ on $\wedge^kT_UM$ over the flow $X_t$ of $X$ on $U$,
that is, for any $k$ choice, $u_{1},u_{2},\cdots,u_{k}$ of vectors in
$T_x M, x\in U$ and $t\in\RR$ such that $X_t(x)\in U$ we set
\begin{align*}
  (\wedge^{k} DX_t)
  \cdot(u_{1}\wedge u_{2} \wedge \cdots \wedge u_{k} )
  =(DX_t\cdot
  u_{1})\wedge (DX_t\cdot
  u_{2})\wedge \cdots \wedge (DX_t\cdot
  u_{k})
\end{align*}
see e.g. \cite[Chapter 3, Section 2.3]{arnold-l-1998} or
\cite{Winitzki12} for more details and standard results on
exterior algebra and exterior products of linear operators.

\begin{remark}
  \label{rmk:rvolumeflow}
  In particular, if $DX_{t}(u_{i}) = v_{i}(t) = v_{i}$, where $G$ is a
  subspace of $T_xM$ spanned by $u_{1},\cdots,u_{r} \in T_xM$, then
  $H=DX_t(G)$ is spanned by $v_{1},\cdots,v_{r}$. Thus
\begin{align*}
  |\det(DX_{t}|_{G})|
  &=
    \vol(DX_{t}(u_{1}),\cdots,DX_{t}(u_{r}))
  \\
  &=\
    \|DX_{t}(u_{1}) \wedge \cdots \wedge DX_{t}(u_{r})\|
    =
    \|\wedge^{r} DX_{t}(u_{1} \wedge \cdots \wedge u_{r})\|.
\end{align*}
\end{remark}

\subsubsection{Conditions for sectional expansion}
\label{sec:condit-section-expan}

  Let $\Gamma$ be a compact invariant set of $X$ such that
  it exhibits a continuous $DX_t$-invariant splitting
  $E \oplus F$ on a subbundle of $T_{\Gamma}M$, with
  $E_x \neq \{0\}$, $F_x \neq \{0\}$ for every
  $x \in \Gamma$.

  If $\{e_1,\dots,e_\ell\}$ is a basis for $E$ and $\{f_1,\dots,f_m\}$
  is a basis for $F$ (where the $e_i$ amd $f_k$ are vector fields over
  $\Gamma$), then $\wt F=\wedge^kF$ generated by
  $\{f_{i_1}\wedge\dots\wedge f_{i_k}\}_{1\le i_1<\dots<i_k\le m}$ is
  a subbundle of $\wedge^kT_\Gamma M$ which naturally is
  $\wedge^k DX_t$-invariant by construction.

In addition, $\tilde E=E\wedge\big(\wedge^{k-1}(E\oplus F)\big)$
generated by all the elementary $k$-vector of the form
\begin{align*}
e_{i_1}\wedge\dots\wedge e_{i_h}\wedge f_{j_1}\wedge\dots\wedge
f_{j_{k-h}}
\end{align*}
for $1\le h\le k $ and $1\le i_1<\dots<i_h\le k$,
$1\le j_1<\dots<j_{k-h}\le m$ (with no $f$ vectors if $h=k$), is
also a subbundle of $\wedge^kT_\Gamma M$ which is
$\wedge^kDX_t$-invariant.

Moreover, $\wt E\oplus \wt F=\wedge^k T_\Gamma M$ gives
a splitting of the $k$th exterior power of the subbundle $T_\Gamma
M$. We stress that each subbundle $\wt{E}_x, \wt{F}_x$ at any given
$x\in\Gamma$ is generated by basis of both $E_x$ and $F_x$ in a
pointwise construction.

Let us assume that $\dim F_x \geq 2$ is constant for $x\in\Gamma$ and
there are positive constants $C , \lambda$ such that for every
$x \in \Gamma$ and every $p$-dimensional linear subspace
$L_x \subset F_x$ one has
$\vert \det (DX_{-t} \vert_{L_x})\vert < C e^{-\la t}, \textrm{ for
  all } t>0$. That is, $F$ is $p$-sectionally expanding.

By the above discussion, this means that
$\|\wedge^{k} DX_{-t}(u_{1} \wedge \cdots \wedge u_{p})\| \leq C
e^{-\la t} \|u_{1} \wedge \cdots \wedge u_{p}\|$ for any orthonormal
basis $u_{1},\cdots,u_{p}$ of $F_{x}$.

We are going to show that this inequality for a certain basis implies
that $\|\wedge^{k} DX_{-t}(w)\| \leq C e^{-\la t} \|w\|$ for all
$w \in \wedge^{k} F_{x}$. Thus, \emph{it is enough to obtain the
  relevant bounds on certain elementary $k$-vectors of the
  $k$-exterior power extension.}

We recall that given a linear isomorphism
$A:(E,\langle\cdot,\cdot\rangle)\to (F,[\cdot,\cdot])$
between finite dimensional Euclidean vector spaces, the
Singular Value Decomposition provides:
\begin{itemize}
\item $d=\dim E=\dim F$ non-negative reals
  $0\le\lambda_1\le\ldots\le\lambda_d$;
\item a pair of orthonormal basis $\U=\{u_1,\dots, u_d\}$ of
  $E$ and $\V=\{v_1,\dots,v_d\}$ of $F$ such that
  $Au_i=\lambda_iv_i, i=1,\dots,d$.
\end{itemize}
We say that $\U$ is a \emph{singular basis} and
$\lambda_1,\dots,\lambda_d$ are the corresponding
\emph{singular values} for the given linear transformation.
Let us define
$\U^p=\{u_{i_1}\wedge\cdots\wedge u_{i_p}: 1\le
i_1<\cdots<i_p\le d\}$ for $1<p\le d$ and observe that
$\U^p$ is an orthonormal basis for $\wedge^pE$ and,
likewise, $\V^p$ is an orthonormal basis for $\wedge^pF$,
with respect to the induced exterior inner
products. Moreover
\begin{align*}
  \wedge^pA\cdot u_{i_1}\wedge\cdots\wedge u_{i_p}
  =
  (\lambda_{i_1}\cdots\lambda_{i_p})
  \cdot v_{i_1}\wedge\cdots\wedge v_{i_p}
\end{align*}
for all relevant indices, and the number of elements in
$\U^p$ is $\binom{d}{p}=\dim \wedge^pE$. Hence, $\U^p$ is a
singular basis for the operator $\wedge^pA$ and for each
$u\in\U^p$ the corresponding singular values are
$ \lambda(u) = \lambda(u_{i_1}\wedge\cdots\wedge u_{i_p}) =
\lambda_{i_1}\cdots\lambda_{i_p}.  $

\begin{remark}
  \label{rmk:singularbasisplit}
  If $E=E_1\oplus E_2$ and $F=F_1\oplus F_2$ are orthogonal
  splittings of $(E,\langle\cdot,\cdot\rangle)$ and
  $(F,[\cdot,\cdot])$ preserved by $A$, that is,
  $A(E_i)=F_i, i=1,2$, then given singular basis $\U_i$ for
  $A\mid_{E_i}, i=1,2$ we have that $\U_1\cup\U_2$ is a
  singular basis for $A$. Likewise, $(\U_1\cup\U_2)^p$ is a
  singular basis for $\wedge^pA$ for any $1<p\le d$.
\end{remark}

\begin{lemma}\label{le:aux1}
  Let $A:(E,\langle\cdot,\cdot\rangle)\to (F,[\cdot,\cdot])$
  be a linear isomorphism between $d$-dimensional Euclidean
  vector spaces, let $\U$ be a singular basis for $A$ and
  fix $1<p\le d$. If there exists $\xi>0$ satisfying
  $\|\wedge^{p} A \cdot u\| \leq \xi$ for $u \in \U^p$,
  then
  $\|\wedge^{p} A\| =\sup_{w\in\wedge^{p}E} \frac{\|\wedge^pA\cdot
  w\|}{\|w\|} \leq \xi$. Moreover, if $p<d$, then
  $\|\wedge^{p+1}A\|\le\xi\|A\|$.
\end{lemma}

\begin{proof}
  Fixing $w\in\wedge^pE$ we have
  $w=\sum_{u\in\U^p}\beta(w,u)u$ for some scalars
  $\beta(w,u)$ so that $\sum_{u\in\U^p}\beta(w,u)^2=\|w\|^2$.
  Thus, by definition of singular basis
  \begin{align*}
    \|\wedge^pA\cdot w\|^2
    &=
      \|\sum_{u\in\U^p}\beta(w,u)\cdot\big(\wedge^pA\cdot u\big)\|^2
      =
      \sum_{u\in\U^p}\beta(w,u)^2\|\wedge^pA\cdot u\|^2
    \\
    &\le
      \sum_{u\in\U^p}\beta(w,u)^2\xi^2
      =
      \xi^2\|w\|^2.
  \end{align*}
  Since $w\in\wedge^{p}E$ was arbitrarily chosen, the proof of
  the first statement of the lemma is complete.

  Finally, note that the assumption implies that
  $\lambda(u)\le\xi$ for all $u\in\U^p$. Then the
  singular values of $\wedge^{p+1}A$ are products of
  singular values of $A$ by some $\lambda(u)$ for
  $u\in\U^p$. Since $\lambda_i\le\|A\|$ we get
  $\|\wedge^{p+1}A\|\le
  \|A\|\cdot\|\wedge^{p}A\|\le\xi\|A\|$, completing the
  proof of the lemma.
\end{proof}

\subsection{Dominated/partial/sectional-hyperbolic splittings and
  exterior powers}\label{le:exteriorprod}

In \cite{arsal2015a}, the first and last authors together proved the
following relation between a dominated splitting and its exterior
power.

\begin{theorem}\label{bivectparthyp2} \cite[Theorem A]{arsal2015a}
  Assume that $T_{\Gamma}M = E \oplus F$ is a $DX_{t}$ invariant
  splitting with $\dim F=k\ge2$. The splitting
  $T_{\Gamma}M = E\oplus F$ is dominated for $DX_t$ if, and only if,
  $\wedge^k T_{\Gamma}M = \widetilde E\oplus \widetilde F$ is a
  dominated splitting for $\wedge^k DX_t$.
\end{theorem}
Hence, the existence of a dominated splitting
$T_\Gamma M=E_\Gamma\oplus F_\Gamma$ over the compact
$X_t$-invariant subset $\Gamma$, is equivalent to that the bundle
$\wedge^{k}T_\Gamma M$ admits a dominated splitting with
respect to
$\wedge^{k}DX_t: \wedge^{k}T_\Gamma M \to \wedge^{k}T_\Gamma
M $.

As a consequence, they obtain the next characterization of
three-dimensional singular sets.

\begin{corollary}\cite[Corollary 1.5]{arsal2015a}
  \label{corA1} Assume that $T_{\Gamma}M = E \oplus F$ is a $DX_{t}$
  invariant splitting.  Suppose that $M$ has dimension $3$, $E$ is
  uniformly contracted by $DX_{t}$ and $\dim F=2$. Then $E \oplus F$
  is a singular hyperbolic splitting for $DX_{t}$ if, and only if,
  $ \widetilde E \oplus \widetilde F$ is a partially hyperbolic
  splitting for $\wedge^{2} DX_t$ such that $\widetilde F$ is
  uniformly expanded by $\wedge^{2} DX_t$.
\end{corollary}

The result below generalizes Corollary~\ref{corA1} to
arbitrary $n$ and $k$. 

\begin{lemma}\cite[Lemma 3.4]{salvinc2017} \label{bivectparthyp1}
  Assume that $T_{\Gamma}M = E \oplus F$ is a $DX_{t}$-invariant
  splitting with $\dim M=n$ and $\dim F=k\ge2$.  The subbundle
  $F_{\Gamma}$ is volume expanding by $DX_t$ if, and only if,
  $\widetilde F$ is uniformly expanded by $\wedge^{k} DX_t$.

  In particular, $E \oplus F$ is a singular-hyperbolic
  splitting, where $F$ is volume expanding for $DX_{t}$ if,
  and only if, $ \widetilde E \oplus \widetilde F$ is a
  partially hyperbolic splitting for $\wedge^{k} DX_t$ such
  that $\widetilde F$ is uniformly expanded by
  $\wedge^{k} DX_t$.
\end{lemma}

\begin{corollary}\label{bivectparthyp} Assume that
  $T_{\Gamma}M = E \oplus F$ is a $DX_{t}$ invariant
  splitting. Suppose that $E$ is
  uniformly contracted by $DX_{t}$. Then  $E \oplus F$ is a
  singular-hyperbolic splitting for $DX_{t}$ if, and only
  if, $ \widetilde E \oplus \widetilde F$ is a partially
  hyperbolic splitting for $\wedge^{k} DX_t$ such that
  $\widetilde F$ is uniformly expanded by $\wedge^{k} DX_t$.
\end{corollary}

Let $M$ be an $m$-dimensional Riemannian manifold  with
$\langle \cdot,\cdot \rangle$ inner product in
$T_{\Gamma}M$, and $\langle \cdot,\cdot \rangle_{*}$ the
inner product in $\wedge^{k} T_{\Gamma}M$ induced by
$\langle \cdot,\cdot \rangle$ where
$\wedge^{k} T_{\Gamma}M = \bigcup_{x \in \Gamma} \wedge^{k}
T_{x}M$. So for $x \in \Gamma$, we have that
$\langle \cdot,\cdot \rangle$ is defined on $T_{x}M$, and
$\langle \cdot,\cdot \rangle_{*}$ is defined on
$\wedge^{k} T_{x}M$.

\begin{lemma} \label{lemma189} Let $N$ be a $n$-dimensional
  vector bundle. Then, for each inner product
  $[\cdot,\cdot]_{*}$ in $\wedge^{n-1} N$ there exists an
  inner product $[\cdot,\cdot]$ on $N$ such that
  $[\cdot,\cdot]_{*}$ is induced by $[\cdot,\cdot]$.
\end{lemma}

\begin{proof}
  Let $N$ be a $n$-dimensional vector bundle with an inner
  product $\langle \cdot,\cdot \rangle$ on $N$, and
  $\langle \cdot,\cdot \rangle_{*}$ the inner product in
  $\wedge^{n-1} N$ induced by
  $\langle \cdot,\cdot \rangle$ (see Remark \ref{remark3.1}).

  Take $[\cdot,\cdot]_{*}$ an arbitrary inner product in
  $\wedge^{n-1} N$. Using that $[\cdot,\cdot]_{*}$ and
  $\langle \cdot, \cdot \rangle_{*}$ are inner products in
  $\wedge^{n-1} N$ there exists a vector bundle
  isomorphism $J: \wedge^{n-1} N \to \wedge^{n-1} N$
  such that $[u,v]_{*} = \langle J(u), J(v) \rangle_{*}$ for $u,v \in \wedge^{n-1} N$.

  Define $\varphi : GL(N) \to GL(\wedge^{(n-1)} N)$ by
  $A \mapsto \wedge^{n-1} A$ and note that $\varphi$ is an
  injective linear homomorphism. Hence, due to the
  dimensions of the spaces, $\varphi$ is a linear
  isomorphism.  Thus, there exists $A \in GL(N)$ such that
  $\wedge^{n-1} A = J$. This implies that $[u,v]_{*} = \langle \wedge^{n-1} A(u),  \wedge^{n-1} A(v) \rangle_{*}$ for $u,v \in \wedge^{n-1} N$.

We define $[[x,y]] :=\langle A(x), A(y) \rangle$ for
$x, y \in N_z$ and $z \in \Gamma$. Denote also by $[[\cdot, \cdot]]$ the inner product  over  $\wedge^{n-1} N$ induced by the inner product $[[\cdot, \cdot]]$ over $N$ (see Remark \ref{remark3.1}). 
Then if
$u = u_{1} \wedge \cdots \wedge u_{n-1}$ and
$v = v_{1} \wedge \cdots \wedge v_{n-1}$ we get
$[[u,v]] = \det([[u_{i},v_{j}]])_{n-1\times n-1} =
\det(\langle Au_{i}, Av_{j} \rangle )_{n-1\times n-1} $ and
so, on the one hand, we obtain
$$
[[u,v]] = \det(\langle
A(u_{i}),A(v_{j})\rangle)_{n-1\times n-1} = \langle
\wedge^{n-1} A (u), \wedge^{n-1} A (v) \rangle_{*}.
$$
On the other hand, we also have
$
[u,v]_{*}
=
\langle J(u), J(v) \rangle_{*}
=
\langle \wedge^{n-1} A (u), \wedge^{n-1} A (v) \rangle_{*}.
$
Therefore, $[\cdot,\cdot]_{*} = [[\cdot,\cdot]]$ on all
$(n-1)$-vectors and so by bilinearity equality holds for all
elements of $\wedge^{n-1}N$ and we are done.
\end{proof}


\section{Proofs of main results}\label{sec:proof}

We are now able to prove our main theorems.

\subsection{Existence of singular adapted metric}
\label{sec:existence-singul-ada}

Here we prove Theorem~\ref{mthm:singadaptmetric}.

Let $\Gamma$ be a compact invariant subset of $X$ admitting
a singular-hyperbolic splitting
$T_{\Gamma}M = E^{s} \oplus E^{c}$ with rate given by the
constant $\lambda>0$; see
Definitions~\ref{def:domination},~\ref{def:sechypset}
and~\ref{def:singhypset}.

Let $Y$ be the north-pole-south-pole vector field on the one-sphere
$\sS^{1}$, which we view as $\sS^1=\RR/\ZZ$, with one sink (the south
pole, the point $S=0 \in \sS^{1}$) and one source (the north pole) and
no other recurrent points. Thus $Y$ can be seen as a periodic smooth
function $f_0:\sS^1\to\RR$.

We take the skew-product vector field $Z$ on $M\times\mathbb{S}^1$
given by
$(x,\theta)\mapsto (X(x),\mu(x)Y(\theta))\in T_xM\times T_t\sS^1$,
where $\mu:M\to\RR$ is a smooth function to be specified. Let
$Z_t:M\times\sS^1\circlearrowleft$ be the flow generated by $Z$ and
consider the compact $Z$-invariant subset
$\wt{\Gamma}=\Gamma\times\{S\}$.


Note that $T_{\wt{\Gamma}}(M\times \sS^1)=E^{s} \oplus F \oplus E^{c}$
where $F=T_S\sS^1$ is a one-dimensional subspace, $\dim E^{s} = d_s$ and
$\dim E^{c} = d_c$, $d_s+d_c=m$. Indeed, the flow
$(x,\theta,t)\mapsto Z_t(X,\theta)=(X_t(x),Y_{t,x}(\theta))$, where
\begin{align*}
  Y_{t,x}(\theta)=\theta+\int_0^t\mu(X_s(x))f_0\big(Y_{s,x}(\theta)\big)\, ds,
  \quad
  (x,\theta)\in M\times\sS^1, t\in\RR,
\end{align*}
is uniquely defined and preserves $\{0\}\times T\sS^1$. In fact, any curve
$v(s)=\{x\}\times\theta(t)$ for a given fixed $x\in M$ and smooth
$\theta:I\to\sS^1$ is sent to
$Z_t(v(s))=\big(X_t(x),Y_{t,x}(\theta(s))\big)$, which is again a
curve in $\{X_t(x)\}\times\sS^1$; where $I$ denotes the unit interval
$[0,1]$.

Denoting by $[\cdot,\cdot]$ an inner product over
$T_\Gamma M$ and by $\cdot$ some inner product on $\sS^1$,
we define for
$v= v_{1}+f_{1},u = u_{1} +f_{2} \in E^{s} \oplus F \oplus
E^{c} $, where $v_{1},u_{1} \in E^{s} \oplus E^{c}$ and
$f_{1},f_{2} \in F$, the inner product on
$T_{\Gamma\times\{S\}}(M\times\sS^1)$ by
$ \langle v,u \rangle = [ v_{1},u_{1}] + f_{1}\cdot f_{2}.$

\begin{remark}\label{rmk:ortonorm}
  \begin{enumerate}
  \item The choice of the inner product ensures that the subbundle $F$
    becomes orthogonal to $E^s\oplus E^c$.  
  \item We also have that
    $\|DZ_t(x,S)\mid_F\|=\exp\left[-\alpha_0\int_0^t\mu(X_s(x))\,ds\right]
    =\mm(DZ_t(x,S)\mid_F)$\footnote{Here
      $\mm (A) = \inf_{\|u\|=1} \| A (u) \|$ for a morphism $A$ of
      normed vector spaces.}, $x\in M, t\in\RR$, where
    $-\alpha_0=f_0'(0)<0$, for any Finsler on $M\times\{S\}$, by
    construction of $Z_t$.
  \end{enumerate}
\end{remark}


We are going to show  that there exists the dominated splitting $T_{\wt{\Gamma}}(M\times\sS^1) = E^{s} \oplus (F
\oplus E^{c})$ over the compact invariant set $\wt{\Gamma}$ of $Z$. To prove this result we use the following auxiliary results.

\subsubsection{Subadditive functions of the orbits of a flow and exponential growth}
\label{sec:subadd-functi-orbits}

We say that a family of functions $\phi:\RR\times\Gamma\to\RR$ is
\emph{subadditive} if
\begin{align*}
  \phi(t+s,x)\le\phi(s,X_t(x))+\phi(t,x),\quad
  \text{for all}\quad t,s\in\RR, x\in\Gamma.
\end{align*}

\begin{lemma}{\cite[Lemma 4.12]{AraPac2010}}
  \label{le:subdiff}
  Let $\phi:\RR\times \Gamma\to\RR$ be a subadditive function for the flow
  of $X$ satisfying $\phi(0,x)=0$ and
  $D(x):=\limsup_{h\to0} (\phi(h,x)/h)<\infty$.  Then
  $\partial_h \phi(h,x)\mid_{h=0}=D(x)=\lim_{h\to0}\phi(h,x)/h$.
\end{lemma}
We say that $D(x)$ is the infinitesimal generator of $\phi$ and we can
deduce the following useful relation; see e.g.~\cite[Section
4.3.1.1]{AraPac2010}:
\begin{align*}
  \phi(t,x)=
  \int_0^t D\big(X_s(x)\big) \, ds, \quad x\in\Gamma, t\in\RR.
\end{align*}
Hence, for each fixed $x\in\Gamma$, the function of $t$ above is in
fact \emph{additive}. The following will be very useful.

\begin{lemma}
  \label{le:subaddinduction}
  If $\phi,\psi:[0,+\infty)\times\Gamma\to\RR$ be subadditive
  functions on an $X$-invariant compact subset $\Gamma$ such that
  there exists $a>0$ satisfying $\phi(t,x)\le\psi(t,x)$ for all
  $0\le t\le a$ and $x\in\Gamma$. Then $\phi\le\psi$.
\end{lemma}

\begin{proof}
  For any given $s\ge0$ we can write $s=[s]+t$ with
  $[s]=\sup\{n\in\ZZ^+:na\le s\}$ and $0\le t\le a$. Then let $m=[s]\in\ZZ$
  and by assumption
  \begin{align*}
    \phi(t,x)=\int_0^tD(X_ux)\,du\le\int_0^tE(X_ux)\,du=\psi(t,x),\quad x\in\Gamma;
  \end{align*}
  where $D,E$ are the infinitesimal generators of $\phi,\psi$. Hence
  for any fixed $x\in\Gamma$
  \begin{align*}
    \phi(s,x)
    &=
      \int_0^sD(X_ux)\,du
      =
      \sum_{i=0}^{(n-1)a}\int_0^aD(X_{ia+u}x)\,du +
      \int_n^tD(X_{na+u}x)\,du
    \\
    &\le
      \sum_{i=0}^{(n-1)a}\int_0^aE(X_{ia+u}x)\,du +
      \int_{na}^tE(X_{na+u}x)\,du
      =
      \int_0^sE(X_ux)\,du
      =\psi(s,x)
  \end{align*}
  as stated.
\end{proof}

We apply this to $\phi(t,v)=\log\|DZ_t\cdot v\|$ for
$v\in T^1(M\times\sS^1)$ obtaining a continuous function
$D^s:T^1_{\wt{\Gamma}} (M\times\sS^1)\cap E^s\to\RR$, where $T^1$
denotes the unit tangent bundle, such that
\begin{align*}
  \log\|DZ_t v_s\| =\int_0^tD^s(\phi_uv_s)\,du,
  \quad\text{with}\quad \phi_uv_s=\frac{DZ_u v_s}{\|DZ_u v_s\|}
  \qand v_s\in E^s_x, \|v_s\|=1
\end{align*}
see \cite[Theorem 1.11]{Ara2020}.
Analogously we obtain continous
functions $D^c, D^s_F$ on $T^1_{\wt\Gamma}(M\times\sS^1)$ satisfying
\begin{align*}
  \log\|DZ_{-t}\phi_tv_c\|
  &=
\int_0^tD^c(\phi_uv_c)\,du\quad\text{with}\quad v_c\in E^c_x,\|v_c\|=1; \qand
\\
\log\|DZ_tv\|
&=
\int_0^t D^s_F(\phi_uv)\,ds\quad\text{with}\quad v\in (E^s\oplus F)_x, \|v\|=1.
\end{align*}
Note that in Remark~\ref{rmk:ortonorm}(2) we already obtained an
expression for the infinitesimal generator along the one-dimensional
bundle $F$.

\subsubsection{Dominated splitting for the skew-product flow}
\label{sec:dominated-splitting}

We are now ready to prove the following result.

\begin{lemma}\label{le:dominations}
  There exist $C, \xi>0$ and a smooth function $\mu:M\to\RR$ such that
  for the skew-product vector field $Z$ on $M\times\sS^1$, for all
  $x \in \widetilde{\Gamma}$ and $t>0$
  \begin{enumerate}
  \item
    $\|DZ_t\mid_{E_{x}^{s}} \|\cdot \|DZ_{-t}\mid_{F \oplus
      E_{Z_tx}^{c}}\| \le C e^{-\xi t}$; and
  \item
    $\|DZ_t\mid_{E_{x}^{s}\oplus F} \|\cdot
    \|DZ_{-t}\mid_{E_{Z_t x}^{c}}\| \le C e^{-\xi t}$;
  \end{enumerate}
  where
  $\|w\|= \langle w,w \rangle^{1/2}, w\in
  T_{\widetilde{\Gamma}}(M\times\sS^1)$.
\end{lemma}

\begin{proof}[Proof of Lemma~\ref{le:dominations}] For the first item,
  take $v_{s} \in E^{s}_x$ and
  $w = f + v_{c} \in (F \oplus E^{c})_{Z_t x}$, where
  $f \in F_{Z_t x}$ and $v_{c} \in E_{Z_t x}^{c}$ and all these
  vectors are unitary. Then
  \begin{align*}
    \|DZ_t &(v_{s}) \|^{2}
    \cdot \| DZ_{-t} (f + v_{c})\|^{2}
      =
      \|DZ_t (v_{s}) \|^{2} \cdot (\| DZ_{-t} (f)\|^{2} +
      \| DZ_{-t} (v_{c})\|^{2} )\\
    &=
      \|DZ_t (v_{s}) \|^{2} \cdot \| DZ_{-t} (f)\|^{2} +
      \|DZ_t (v_{s}) \|^{2} \cdot \| DZ_{-t} (v_{c})\|^{2}\\
    &=
      \| DZ_{-t} (f)\|^{2}\exp\big[{2\int_0^tD^s\phi_u v_s\,du}\big]
      +\exp\big[{2\int_0^t (D^s\phi_u v_s+D^c\phi_u v_c)\,du}\big]
    \\
    &=e^{-2\int_0^{-t}\alpha_0\mu\circ
      Z_u(Z_tx)\,ds+2\int_0^tD^s\cdot\phi_u v_s\,du}
      +e^{2\int_0^t (D^s\cdot\phi_u v_s+D^c\cdot\phi_u v_c)\,du}
    \\
    &\le
      e^{2\int_0^tD^s\cdot\phi_u v_s\,du}
      (e^{2\int_0^{t}\alpha_0\mu\circ Z_u x\,du}+e^{2\int_0^t D^c\cdot\phi_u v_c)\,du})
    \\
    &=
      \exp\big[{2\int_0^t (D^s\cdot\phi_u v_s+D^c\cdot\phi_u v_c)\,du}\big]+
      \exp\big[{2\int_0^t (D^s\cdot\phi_u v_s+\alpha_0\mu\circ Z_u x)\,du}\big].
  \end{align*}
  Now note that from domination we get
  \begin{align*}
    \exp\left[{2\int_0^t (D^s\cdot\phi_u v_s+D^c\cdot\phi_u v_c)\,du}\right]
    &=
    \|DZ_t\cdot v_s\|^2\cdot\|DZ_{-t}\phi_tv_c\|^2
    \\
    &\le
    \|DZ_t\mid E^s_x\|^2\cdot\|DZ_{-t}\mid E^c_{X_tx}\|^2
    \le
    Ke^{-2\lambda t}.
  \end{align*}
  So, it is enough to choose $\mu$ so that
  $\int_0^t (D^s\phi_uv_s+\alpha_0\mu\circ Z_sx)\,ds\le-\xi t$ for
  some $\xi>0$ and all choices of unitary vectors $v_s\in E^s_x$.  If
  $\mu:\wt\Gamma\to\RR$ is a smooth function such that
  \begin{align}\label{eq:ineqmu}
    \sup_{\|v_s\|=1}D^s(v_s)+\zeta
    < -\alpha_0\mu(x)
    < \sup_{\|v_s\|=1}D^s(v_s)+2\zeta, \quad x\in\wt\Gamma
  \end{align}
  and $\zeta>0$ is sufficiently small, then
  $ \int_0^t (D^s\phi_uv_s+\alpha_0\mu(Z_ux))\,du \le -\zeta t$.  This proves
  item (1) with $C=K$, $\xi=\zeta$ and $\mu$ chosen as above.

  For the second item, note that for all $t>0$, $v_s\in E^s_x$ with
  $\|v_s\|=1$ and $x\in\wt\Gamma$
  \begin{align*}
    \log\|DZ_t\mid_F\|
    =
    \int_0^t-\alpha_0\mu(Z_ux)\,du
    >
    \int_0^t(D^s\phi_uv_s+\zeta)\,du
    =\zeta t + \log\|DZ_t v_s\|
  \end{align*}
  which implies
  \begin{align}\label{eq:F_Es}
    \log\|DZ_t\mid_F\|
    \ge
    \zeta t + \log\|DZ_t v_s\|
    \qand
    \|DZ_t\mid_{(E^s\oplus F)_x}\|=\|DZ_t\mid_F\|.
  \end{align}
  Hence we obtain $D^s_F=-\alpha_0\mu$. Thus 
\begin{align*}
  \| DZ_t |_{(E^{s}\oplus F)_x} \|\cdot\|( DZ_t |_{E^{c}_{x}})^{-1}\|
  &=
    \|DZ_t\mid_F\|\cdot\| DZ_{-t}\mid_{E^{c}_{Z_tx}}\|
  \\
  &\le
    \exp\left[\int_0^t(-\alpha_0\mu(Z_ux)+\sup_{\|v_c\|=1}D^c\phi_u v_c) \, du\right]
  \\
  &\le
    \exp
    \left[
    2\zeta
    t+\int_0^t(\sup_{\|v_s\|=1}D^s\phi_uv_s+\sup_{\|v_c\|=1}D^c\phi_uv_c)\,
    du
    \right]
  \\
  &\le
    e^{2\zeta t}  \| DZ_t \mid_{E^s_x} \|\cdot\|
    DZ_{-t}\mid_{E^c_{Z_tx}}\|
  \le
    e^{2\zeta t} K e^{-\lambda t}.
\end{align*}
We just have to choose $0<2\zeta<\lambda$ to conclude the statement of
item (2) with $\xi=\lambda-2\zeta$ and $C=K$.

This concludes the the proof of the Lemma~\ref{le:dominations} with
$C=K$ and $\xi=\min\{\zeta,\lambda-2\zeta\}$.
\end{proof}

This shows that there exists the dominated splitting
$ T_{\wt{\Gamma}}(M\times\sS^1) = E^{s} \oplus (F
\oplus E^{c})$ over the compact invariant set $\wt{\Gamma}$
of $Z$. In particular, this also shows that
$ F \oplus E^{c}$ is a partially hyperbolic splitting with
respect to $DZ_t\mid (F \oplus E^{c})$, such that the volume
along $E^c$ is uniformly expanded. Thus, $F\oplus E^c$
\emph{is a codimension-one singular-hyperbolic splitting}
of the subbundle $F\oplus E^c$ of
$T_{\wt{\Gamma}}(M\times\sS^1)$.

\begin{remark}\label{rmk:wedgesplit}
  We can write
  $\wedge^{d_c}(E^s\oplus F\oplus E^c)=\widetilde{E}\oplus
  [F\wedge(\wedge^{d_c-1}E^c)]\oplus \wedge^{d_c} E^c$,
  where $\widetilde E$ is formed by non-null exterior
  vectors involving always some vector from $E^s$.
\end{remark}


We denote $\widetilde F=F\wedge(\wedge^{d_c-1}E^c)$ and
$\widetilde{E^c}=\wedge^{d_c} E^c$ in what follows. Then
from Theorem~\ref{bivectparthyp2} the splitting
$T_{\wt{\Gamma}}(M\times\sS^1) = E^{s} \oplus (F \oplus
E^{c})$ is dominated for $DZ_t$ if, and only if,
$\wt{E}\oplus\wt{F}\oplus\wt{E^c}$ is a dominated splitting
for $\wedge^{d_c} DZ_t$.  Since we additionally have
$|\det (DZ_{-t} \mid_{E_{x}^{c}}) | \le Ce^{-\xi t}$ for
some $C,\xi>0$ by assumption of singular-hyperbolicity, then
$\wt{F}\oplus\wt{E^c}$ \emph{is a partially hyperbolic
  splitting of the invariant subbundle
  $\wedge^{d_c}(F\oplus E^c)$ such that $\widetilde {E^{c}}$
  is uniformly expanded by}
$\wedge^{d_c} DZ_t|_{ \widetilde{F} \oplus \widetilde
  {E^{c}}}$.

Hence, from \cite[Theorem 1]{Goum07}, there exists an
adapted inner product $[\cdot,\cdot]$ for
$\wedge^{d_c} DZ_t|_{ \widetilde{F} \oplus \wt{E^{c}}}$ over
$\wt{\Gamma}$, that is, there exists $\sigma>0$
satisfying
\begin{align*}
  [\wedge^{d_c}DZ_t\mid_{\widetilde{F}_{x}}] \cdot
  [\wedge^{d_c}DZ_{-t}\mid_{\widetilde {E^{c}}_{Z_t(x)}}] \le
  e^{-\sigma t}
  \quad\text{and}\quad
  [\wedge^{d_c} DZ_{-t} \mid_{ \widetilde {E^{c}}_x}] \leq
  e^{-\sigma t}, \quad \forall t>0, \forall x\in\widetilde{\Gamma}.
\end{align*}
Lemma~\ref{lemma189} now ensures that there exists an inner
product $[[ \cdot , \cdot ]]$ on the subbundle
$ F \oplus E^{c}$ such that $[\cdot, \cdot]$ is induced by
$[[ \cdot , \cdot ]]$. This last inner product satisfies the
conditions of the following result, whose proof we postpone
to the end of this section.

  \begin{lemma}\label{le:final}
    Suppose that there exists an inner product
    $\langle \langle \cdot , \cdot \rangle \rangle$ on an
    invariant subbundle $G$ of dimension $m$ of
    $T_{\Gamma}M$, such that it admits a continuous
    splitting $G=E\oplus F$ and whose induced inner product
    $\langle \langle \cdot , \cdot \rangle \rangle_{*}$ on
    $\wedge^{(m-1)}G$ satisfies for each $x\in\Gamma$ and
    $t>0$
  \begin{align*}
    \|\wedge^{m-1}DX_t\mid_{\tilde{E}_x}\|_{*} \cdot
    \|\wedge^{m-1}DX_{-t}\mid_{\tilde{F}_{X_t(x)}}\|_{*} \le
    e^{-\lambda t}\quad\text{and}\quad
    \|\wedge^{m-1} DX_t \mid_{ \widetilde{F}_x}\|_{*} \geq e^{\lambda
    t};
  \end{align*}
  where $\|\cdot\|$ is the norm induced by
  $\langle \langle \cdot , \cdot \rangle\rangle$.
    Then there exists an
    inner product $\langle \cdot , \cdot \rangle$ in
    $T_{\Gamma}M$ such that for all $t>0$
\begin{enumerate}
\item
  $|DX_t\mid_{E_x}|\cdot|DX_{-t}\mid_{F_{X_t(x)}}|\le
  e^{-\lambda t}$;
\item
  $|\wedge^{m-1}DX_t\mid_{\tilde{E}_x}|_{*}\cdot
  |\wedge^{m-1}DX_{-t}\mid_{\tilde{F}_{X_t(x)}}|_{*}\le e^{-\lambda t}$; and
\item
  $|\wedge^{m-1} DX_t\mid_{ \widetilde F_x}|_{*} \geq e^{\lambda t}$.
\end{enumerate}
where $|\cdot|$ is the norm induced by
$\langle \cdot , \cdot \rangle$.
\end{lemma}

Hence, applying the above lemma to the subbubdle
$G=F\oplus E^c$ with $F=E^c$ and $E=F$, we obtain that there
exists a singular adapted inner product $[[\cdot, \cdot]]_*$
on $F \oplus E^{c}$, that is, satisfying
\begin{align}\label{eq:dolemafinal}
 [[DZ_t\mid_{F_{x}}]]_* \cdot
  [[DZ_{-t}\mid_{E^{c}_{Z_t(x)}}]]_* \le
  e^{-\sigma t}
  \qand
  [[\wedge^{d_c} DZ_{-t} \mid_{ \widetilde {E^{c}}_x}]]_*
    \leq e^{-\sigma t}
\end{align}
for some $\sigma>0$ and for all $t>0$,
$x\in\widetilde{\Gamma}$.

\subsubsection{Adapting the domination on the stable
  direction}
\label{sec:adapting-dominat-sta}

We need the following lemma to find a metric which is adapted to the
uniform contraction along the stable direction.

  \begin{lemma}
    \label{le:adaptstable}
    Let $H=E^s\oplus F\oplus E^c$ be a $DZ_t$-invariant and continuous
    splitting of a vector subbundle such that $E^s$ is uniformly
    $(K,\lambda)$-contracted; the subbundle $F\oplus E^c$ admits an
    inner product $\langle\cdot,\cdot\rangle$ inducing a metric
    $|\cdot|$ satisfying items $(1), (2)$ and $(3)$ of the conclusion
    of Lemma~\ref{le:final}; and $F$ satisfies
    Remark~\ref{rmk:ortonorm} and Lemma~\ref{le:dominations}. Then
    there exists a singular adapted metric for the subbundle
    $E^s\oplus E^c$.
  \end{lemma}

  \begin{proof}
    We follow \cite{Goum07} and define a norm on $E^{s}$ by
    $|u_{s}| = \int_{0}^{+\infty} e^{\beta\xi r} \|DX_{\beta r}
    u_{s}\|\,dr$ for any $u_{s} \in E^{s}$, where $\xi$ is given by
    Lemma~\ref{le:dominations} and $\beta=K/(\lambda-2\xi)$.  It is
    easy to check that $|DX_{t}u_{s}| \le e^{-\xi t} |u_{s}|$ for all
    $t >0$ and also that this norm is Riemannian, i.e., induced by an
    inner product $(\cdot,\cdot)$.

    From the construction of $F$ and~\eqref{eq:F_Es} we obtain for
    $v\in E^s_x$
    \begin{align*}
      |DX_t v|
      &=
        \int_0^\infty e^{\beta\xi u}\|DX_{\beta u+t}v\|\,du
        \le
        \int_0^\infty e^{\beta\xi u} e^{-\xi(\beta u+t)}
        \|DZ_{\beta u+t}\mid_{F_x}\|\,du
      \\
      &=
        e^{-\xi t}\int_{t/\beta}^\infty\|DZ_{\beta r}\mid_{F_x}\|\,dr.
    \end{align*}
    We need an auxiliary result which we state here but whose proof we
    postpone.
\begin{lemma}
  \label{le:expint}
  There exists $a>0$ so that
  $e^{-\xi t}\int_{t/\beta}^\infty\|DZ_{\beta u}\mid_{F_x}\|\,du
  \le\|DZ_t\mid F_x\|$ for all $0\le t\le a$.
\end{lemma}
Hence we deduce that $|DX_t v|\le\|DZ_t\mid F_x\|$ for each
$0\le t\le a, x\in\wt{\Gamma}$ and some $a>0$. By
Lemma~\ref{le:subaddinduction} we conclude that the inequality can be
extended to all $t\ge0$. Since the norm of $DZ_t\mid_{F_x}$ does not
depend on the chosen Finsler, we can write
    \begin{align*}
      |DX_{t}|_{E^{s}_{x}}| \cdot
      [[DZ_{-t}\mid_{E^{c}_{Z_t(x)}}]]_*
      &\le
        \|DZ_t\mid_{F_x}\| \cdot [[DZ_{-t}\mid_{E^{c}_{Z_t(x)}}]]_*
      \\
      &=
        [[DZ_t\mid_{F_{x}}]]_* \cdot[[DZ_{-t}\mid_{E^{c}_{Z_t(x)}}]]_*
        \le
        e^{-\sigma t}.
    \end{align*}
    We can now define an inner product
    $\langle \cdot,\cdot \rangle$ on
    $T_{\Gamma}M = E^{s} \oplus E^{c}$ by
    $\langle u,v \rangle = ( u_{s}, v_{s}) +
    [[u_{c},v_{c}]]_*$ for all
    $u = u_{s} + u_{c}, v = v_{s}+v_{c} \in T_{\Gamma}M$,
    and let $ \langle \cdot , \cdot \rangle_{*}$ be the
    induced inner product by $\langle \cdot,\cdot \rangle$
    on $\wedge^{d_{c}}T_{\Gamma}M$.

    Then $\langle \cdot , \cdot \rangle$ is an inner product
    in $T_{\Gamma}M$ adapted to the partially hyperbolic
    splitting $E^{s}\oplus E^{c}$ for $DX_t$. Moreover, from
    the definition of the inner product and exterior power,
    it follows that
    $|\det (DX_{-t}\mid_{E_{x}^{c}})| \leq e^{-\sigma t}$
    for all $t>0$.

    We have obtained a singular adapted metric for
    $E^s\oplus E^c$ finishing the proof of
    Lemma~\ref{le:adaptstable}.
  \end{proof}

  This completes the proof of Theorem~\ref{mthm:singadaptmetric}
  except for the proof of Lemma~\ref{le:final} and
  Lemma~\ref{le:expint}, which we present next.

\begin{proof}[Proof of Lemma~\ref{le:final}]
  Let $u \in E_{x}$ and $v \in F_{X_{t}(x)}$ be such that
  $\|u\| = 1= \|v\|$. We observe that for a given fixed $t \in \mathbb{R}$
  \begin{align*}
    \|DX_tu\|\cdot\|DX_{-t} v\|
    =
    \|\wedge^{k}DX_{t}(u \wedge u_{2} \wedge \cdots \wedge
    u_{k})\|
    \cdot
    \|\wedge^kDX_{-t}(v \wedge v_2 \wedge \cdots \wedge v_k)\|
  \end{align*}
  if we choose $u_{2},\cdots,u_{k} \in E_{x}$ and
  $v_{2},\cdots,v_{k} \in F_{X_{t}(x)}$ such that:
  \begin{itemize}
  \item
    $\langle DX_{t}u, DX_{t}u_{j} \rangle = 0$
    for $2\le j\le k$ and
    $\langle DX_{t}u_{j}, DX_{t}u_{l} \rangle =
    \delta_{jl}$ for $2\le j,l \le k$;
  \item
    $\langle DX_{-t}v, DX_{-t}v_j \rangle =
    0$ for $2\le j \le k$ and
    $\langle DX_{-t}v_j, DX_{-t}v_l \rangle
    = \delta_{jl}$ for $2\le j,l \le k$.
  \end{itemize}
  Consequently we obtain
  \begin{align*}
    \|DX_tu\|\cdot\|DX_{-t} v\|
    &\leq
    \|\wedge^kDX_t\| \|\wedge^kDX_{-t}\|
      \|u\wedge u_2\wedge \cdots \wedge u_k\|
      \cdot\|v \wedge v_2\wedge\cdots\wedge v_k\|
    \\
    &\le
      e^{-\lambda t}\|u\wedge u_2\wedge\cdots\wedge u_k\|
      \cdot\|v \wedge v_2 \wedge \cdots \wedge v_k\|.
  \end{align*}
  We note that $\|u_{j}\| \leq \|DX_{-t}(x)\|$
  since $\|DX_t u_j\|=1$ and analogously
  $\|v_j\|\leq \|DX_{t}(X_tx)\|$ since
  $\|DX_{-t}v_{j}\|=1$ for $2\le j \le k$ with
  $\|u\| = \|v\|=1$.

  We now set $R =\max\{1,\kappa_1\}$, where
  $$
  \kappa_1= \sup_{t\in [-1,1]} \sup_{x \in \Gamma} \|DX_{t}(x)\|
  $$
  and define
  $B[0,R]=\{ \eta \in T\Gamma : |\eta| \leq R \}$ a compact
  subset of $T\Gamma$.

  Note that if we set $t \in [-1,1]$, then
  we get $u,u_{2},\cdots, u_{k},v,v_{2},\cdots,v_{k} \in B[0,R]$ in
  the argument above.

  Moreover $\prod_{i=1}^{k} B[0,R]$ is a compact subset of
  $\prod_{i=1}^{k} T\Gamma= \sum_{p \in \Gamma} T_{p}\Gamma
  \times\overset{k}{\cdots}\times T_{p}\Gamma$ and let
  $\mathcal{I}: \prod_{i=1}^{k} T\Gamma \to \wedge^{k}T\Gamma$ be the
  natural injection given by
  $$
  (w_{1},\cdots,w_{k}) \mapsto w_{1} \wedge \cdots \wedge w_{k}.
  $$

  We define $|\cdot| = \gamma \|\cdot\|$ (or
  $\langle\cdot,\cdot\rangle=\gamma^2[[\cdot,\cdot]]$) where $\gamma$
  is a positive number such that
  $$
  \sup\limits_{w \in \prod_{i=1}^{k} B[0,R]} \|\mathcal{I}(w)\| \le\gamma^{-1}.
  $$
  It follows that
  \begin{align*}
    |DX_tu|\cdot|DX_{-t} v|
    &=
    \gamma\|\wedge^{k}DX_{t}(u \wedge u_{2} \wedge \cdots \wedge
    u_{k})\|
    \cdot\gamma
    \|\wedge^kDX_{-t}(v \wedge v_2\wedge \cdots\wedge v_k)\|
    \\
    &\le
    e^{-\lambda t}\gamma
    \|u\wedge u_2\wedge\cdots\wedge u_k\|
      \cdot \gamma\|v \wedge v_2 \wedge \cdots \wedge v_k\|
    \leq e^{-\lambda t}
  \end{align*}
  and note that the choice of $\gamma$ does not change any of the
  previous relations involving $\|\cdot\|$. Then for any given fixed
  $t \in [-1,1]$ we have obtained an adapted metric $|\cdot|$ that
  satisfies the statement of Lemma~\ref{le:final}.  For general $t>0$
  the inequality follows by Lemma~\ref{le:subaddinduction} (with
  $a=1$), since the previous inequality holds between subadditive
  functions.
  
  We have obtained a metric $|\cdot|$ satisfying item (1) in
  the statement of Lemma~\ref{le:final}. Analogously, it satisfies
  items (2) and (3) of the statement of the Lemma, and we
  are done.
\end{proof}

\begin{proof}[Proof of Lemma~\ref{le:expint}]
  Let us set
  $f(t)=e^{-\xi t}\int_{t/\beta}^\infty\|DX_{\beta
    u}\mid_{F_x}\|\,du$ and use~\eqref{eq:ineqmu} to estimate
  \begin{align*}
    0\le
    f(t)
    &\le
      e^{-\xi t}\int_{t/\beta}^\infty
      \exp\left(\int_0^{\beta u}\left[\sup_{\|v_s\|=1}D^s(\phi_w
      v_s)+2\xi\right]
      \,dw\right)\,du
    \\
    &\le
    e^{-\xi t}\int_{t/\beta}^\infty e^{2\xi\beta u}\|DX_{\beta u}\mid
    E^s_x\|\,du
    \le
      e^{-\xi t} \int_{t/\beta}^\infty Ke^{\beta(2\xi-\lambda) u}\,du
    \\
    &\le
      e^{-\xi t}\frac{K e^{(2\xi-\lambda)t}}{\beta(\lambda-2\xi)}
      =
      \frac{K}{\beta(\lambda-2\xi)}
      \le1,
  \end{align*}
as long as $\beta\ge K/(\lambda-2\xi)>0$. Then it is easy to see that
\begin{align*}
  f'(t)
  &=
    -\xi f(t)-e^{-\xi t}\|DX_t\mid_{F_x}\|<0, \forall t>0, \quad\text{and}
  \\
  \left(\frac{f(t)}{\|DX_t\mid_{F_x}\|}\right)'
  &=
    \frac{\|DX_t\mid_{F_x}\|\cdot f'(t)+f(t)\cdot\alpha_0\mu(X_tx)}{\|DX_t\mid_{F_x}\|^2};
\end{align*}
so for $t=0$ the last expression is bounded from above by
$-1 +f(0)(\alpha_0\mu(x)-\xi)$. Since $\mu$ is bounded on $\wt\Gamma$
(because $D^s$ is continuous on the compact set $E^s\cap T^1_\Gamma M$),
we can choose $\alpha_0>0$ small enough so that the last expression is
strictly negative for every $x\in\wt{\Gamma}$.

Finally, because $f(0)\le1$, by smoothness and compactness we conclude
that there exists $a>0$ such that $f(t)\le\|DX_t\mid F_x\|$ for all
$0\le t\le a, x\in\wt\Gamma$, completing the proof.
\end{proof}

\subsection{The uniformly hyperbolic setting}
\label{sec:uniformly-hyperb-set}

Here we prove Theorem~\ref{mthm:sectadaptmetricunif}.

Let $\Gamma$ be a uniformly hyperbolic compact invariant
subset for the $C^1$ vector field $X$ and let $\|\cdot\|$ be
an adapted metric, that is, we have~\eqref{eq:Klambda-hyp}
with $K=1$ and we can assume without loss of generality that
$E^s,E^X$ and $E^u$ are orthogonal; see
e.g. \cite{Goum07}.

We define $|s\cdot X(x)|=|s|$ for all $s\in\RR$ and
$x\in\Gamma$, which gives a new norm on the one-dimensional
line bundle $E^X$ such that (recall that $\Gamma$ does not
contain singularities: $X(x)\neq\vec0, x\in\Gamma$)
\begin{align*}
  |DX_t\mid E^X_x|
  =
  \frac{|DX_t X(x)|}{|X(x)|}
  =
  \frac{|X(X_t(x))|}{|X(x)|}=1, \quad x\in\Gamma, t\in\RR.
\end{align*}
Now we redefine
$|u+s\cdot X(x)+v|_*^2=\|u\|^2+|s\cdot X(x)|^2+\|v\|^2, u\in
E^s_x, v\in E^u_x, s\in\RR$ and note that
\eqref{eq:Klambda-hyp} still holds with the same $\lambda$
and $K=1$ and also
\begin{align*}
  \frac{|DX_t\mid E^s_x|_*}{|DX_t\mid E^X_x|_*}=|DX_t\mid
  E^s_x|_*
  \qand
  \frac{|DX_t\mid E^X_x|_*}{|DX_{-t}\mid E^u_{X_tx}|^{-1}}
  =|DX_{-t}\mid E^u_{X_tx}|
\end{align*}
are both bounded by $e^{-\lambda t}$ for all $t>0$. This
shows that $|\cdot|_*$ is still adapted to the uniform
hyperbolicity; and also adapted to the partially hyperbolic
splitting $E^s\oplus(E^X\oplus E^u)$.

\subsubsection{The norm is adapted to
  singular-hyperbolicity}
\label{sec:norm-adapted-singul}

With this norm, given by the inner product
$(\cdot,\cdot)_*$, it is now easy to calculate
\begin{align*}
  |(\wedge^2DX_{-t})(X(x)\wedge v)|_*^2
  &=
    |DX_{-t}X(x)|^2_*|DX_{-t}v|_*^2-(DX_{-t}X(x),DX_{-t}v)^2_*
  \\
  &\le
  e^{-2\lambda t}|v|_*^2
  =
  e^{-2\lambda t}|v|_*^2|X(x)|_*^2
  =
  e^{-2\lambda t}|X(x)\wedge v|_*^2
\end{align*}
for all $x\in\Gamma, t>0$ and all $u\in E^u_x$, since $E^X$ and
$E^u$ are everywhere orthogonal by construction.

For any pair $u,v\in E^u_x$ and given fixed $t>0$, we can
write $u\wedge v=u\wedge\tilde v$ so that $DX_{-t}\tilde v$
is orthogonal\footnote{Set $\tilde v=v-hu$ so
that $h\cdot DX_{-t}u$ is the orthogonal projection of
$DX_{-t}v$ along $DX_{-t}u$.} to $DX_{-t}u$. Thus
\begin{align}\label{eq:secexpu}
  |(\wedge^2DX_{-t})(u\wedge v)|_*^2
  &=
    |(\wedge^2DX_{-t})(u\wedge\tilde v)|_*^2
    =
    |DX_{-t} u|^2_* |DX_{-t} \tilde v|^2_*
  \\
  &\le
    e^{-4\lambda t}|u|_*^2|\tilde v|_*^2
    =
    e^{-4\lambda t}|u\wedge \tilde v|_*^2
    =
    e^{-4\lambda t}|u\wedge  v|_*^2  .
\end{align}
Hence, the norm of each elementary bivector, obtained from
an orthonormal basis of $\wedge^2(E^X\oplus E^c)$ generated
by vectors on $E^X\cup E^c$, is uniformly expanded by
$\wedge^2DX_t$ at a rate of at least $\lambda$ with the norm
$|\cdot|_*$. From Subsection~\ref{sec:condit-section-expan},
Lemma~\ref{le:aux1} and Remark~\ref{rmk:singularbasisplit},
since we may apply the above inequalities to vectors from a
singular basis for $DX_{-t}\mid_{E^X\oplus E^u}$, this is
enough to conclude that
$\|\wedge^2DX_{-t}\mid_{E^X\oplus E^u}\|\le e^{-\lambda t}$
for each $t>0$.

Then this is a sectional-adapted norm to the splitting
$E^s\oplus(E^X\oplus E^u)$. 

The proof of Theorem~\ref{mthm:sectadaptmetricunif} is
complete.

\subsection{The open class of sectional-hyperbolic sets}
\label{sec:open-class-sectional}

Here provide a proof of
Theorem~\ref{mthm:sectadaptmetricsing}.

  Let $\Gamma$ be a sectional-hyperbolic set for a $C^1$
  vector field whose splitting
  $T_\Gamma M=E\oplus E^c\oplus E^u$ is dominated and such
  that: $E$ is uniformly contracted; $E^c$ two-dimensional
  and area expanding; $E^u$ is uniformly expanded and
  $E^c\oplus E^u$ is strongly partially hyperbolic. So there
  are $K>0$ and $\eta>\lambda>0$ such that for all $t>0$
  \begin{enumerate}
  \item $\|DX_{t}|_{E}\| \leq K e^{-\lambda t}$;
  \item
    $\|DX_{t}|_{E}\| \leq K e^{-\lambda t}
    \mm(DX_{t}|_{E^{c}})$;
  \item $\|DX_{t}|_{E^{c}}\| \leq K e^{\lambda t}$ and
    $\mm(DX_t\mid_{E^u})\ge K e^{\eta t}$;
  \item
    $\| \wedge^{2} DX_{-t} \cdot (v_{1} \wedge v_{2})\|
    \le
    \|v_{1} \wedge v_{2}\| K e^{-\lambda t}$ for any
    $v_{1},v_{2} \in E^{c}$.
  \end{enumerate}
  \begin{proof}[Proof of Theorem~\ref{mthm:sectadaptmetricsing}]
  Analogously to the proof of
  Theorem~\ref{mthm:singadaptmetric}, we introduce
  consecutively
  \begin{itemize}
  \item a skew-product flow $Z_t$ on $M\times\sS^1$ admitting an
    invariant subset $\wt\Gamma=\Gamma\times\{S\}$ with a
    singular and sectional-hyperbolic splitting
    $E^{s}\oplus F_0\oplus E^c\oplus E^u$ of
    $T_{\wt{\Gamma}}(M\times\sS^1)$ satisfying
    Remark~\ref{rmk:ortonorm}, with $F_0$ in the place of $F$ and the
    same choice of inner product as in
    Subsection~\ref{sec:existence-singul-ada}; together with
  \item a direct product flow
    $W_t=Z_t\times \bar Y_{t}:
    (M\times\sS^1)\times\sS^1\circlearrowleft$, where the vector field
    $\bar Y$ given by $f_1:\sS^1\to\RR$ is a North-South vector field
    as in the beginning of Subsection~\ref{sec:existence-singul-ada}
    and we choose $\alpha_1=f'(1/2)\in(\lambda,\eta)$, where $1/2=N$
    is the coordinate of the expanding fixed source at the north pole.
\end{itemize}
In this way, we get $DY_{t}\mid T_{N}\sS^1=e^{\alpha_1 t}$ and
$\wh{\Gamma}=\wt{\Gamma}\times\{N\}=\Gamma\times\{S\}\times\{N\}$
admits a partially hyperbolic splitting
$T_{\wh{\Gamma}}=E^{s}\oplus F_0\oplus E^c\oplus F\oplus E^u$ where,
using the inner product introduced in
Subsection~\ref{sec:existence-singul-ada} making $F_0\oplus F$
orthogonal to $E^{s}\oplus E^c\oplus E^u$, we get
\begin{enumerate}
\item $F_0\oplus E^c$ is a sectional- and singular-hyperbolic
  splitting (since $\dim E^c=2$); and
\item $E^c\oplus F$ and $F\oplus E^u$ are both strongly
  partially hyperbolic splittings.
\end{enumerate}
The first item is a consequence of the arguments in the proof of
Theorem A applied to the action of $Z_t$ over the subbundle
$F_0\oplus E^c$; see Lemma~\ref{le:dominations} and the arguments
following its proof.

The second item above is a consequence of the following estimates for
a.e. point with respect to each $W_t$-invariant probability
measure
  \begin{align*}
    \lim_{t\to+\infty}\frac1t\log\big(
    \|DW_t\mid_{E^c_x}\|\cdot\|DW_{-t}\mid_{F_{W_tx}}\|\big)
    &\le\lambda-\alpha_1<0
      \qand
    \\
    \lim_{t\to+\infty}\frac1t\log\big(
    \|DW_t\mid_{F_x}\|\cdot\|DW_{-t}\mid_{E^u_{X_tx}}\|\big)
    &\le\alpha_1-\eta<0;
  \end{align*}
  since
  $DW_t\mid_{F_x}=DY_y\mid
  T_N\sS^1=e^{\alpha_1t}=\mm(DW_t\mid_{F_x})$; coupled
  with~\cite[Theorem 1.9]{Salgado19} providing domination of the
  splittings $E^c\oplus F$ and $F\oplus E^u$. More precisely, we set
  first
  $f_t(x)=\log ( \|DW_t \mid_{E^c_x} \| \cdot \|(DW_{-t}
  \mid_{F_{W_tx} })^{-1}\| )$ and then again
  $f_t(x)=\log (\|DW_t\mid_{F_x}\|\cdot\|DW_{-t}\mid_{E^u_{X_tx}}\|)$
  in the statement of the following result.

  \begin{lemma}\cite[Corollary 4.2]{ArbSal}
    \label{le:sectional-arbieto}
    Let $\{t\mapsto f_t:S\to \mathbb{R}\}_{t\in \mathbb{R}}$
    be a continuous family of continuous functions on the
    $X_t$-invariant compact subset $S$, which is subadditive
    and suppose that there exists a constant $\lambda>0$ so that
    $\int \big(\lim_{t\to+\infty}\frac{1}{t}f_t(x)\big)\,
    d\eta \le-\lambda$ for every $X_t$-invariant probability measure
    $\eta$.  Then there exists a constant $K>0$ such
    that $\exp f_t(x) \leq Ke^{-\lambda t}$ for every
    $x\in S$ and every $t>0$.
  \end{lemma}
  This provides constants
  $C_1,C_2>0$ so that, since
  $\|DW_t\mid_F\|=e^{\alpha_1 t}=\mm(DW_t\mid_F)$ independently
  of the norm, we obtain
  \begin{align}
    \|DW_t\mid_{E^c_x}\|\cdot\|DW_{-t}\mid_{F_{W_tx}}\|
    \le C_1e^{(\lambda-\alpha_1)t}
      &\qand
      \|DW_t\mid_{F_x}\|\cdot\|DW_{-t}\mid_{E^u_{W_tx}}\|
      \le
      C_2e^{(\alpha_1-\eta)t}\nonumber
    \\
    \text{and so}\quad
    \|DW_t\mid_{E^c_x}\|\le C_1e^{\lambda t}
      &\quad\text{and}\quad
    1<C_2 e^{-\eta t}\mm(DW_t\mid_{E^u_x}).\label{eq:strongdom}
  \end{align}
  These relations, together with items (1-3), ensure that
  $E^{s}\oplus F_0\oplus E^c\oplus F\oplus E^u$ is a partially
  hyperbolic splitting: the splitting is dominated;
  $E^{s}\oplus F_0$ is uniformly contracted; and $F\oplus E^u$
  uniformly expanded. Thus, we can assume without loss of
  generality that we have a Riemmanian metric $\|\cdot\|$ on
  $T_{\wh\Gamma}(M\times(\sS^1)^2)$ such that this splitting
  is orthogonal, $C_1=C_2=1$ in~\eqref{eq:strongdom}
  together with
  \begin{align}\label{eq:domF0}
    \|DX_t\mid E^{s}\|=\|DW_t\mid E^{s}\|
    &\le
      e^{-\lambda t}\mm(DW_t\mid F_0)
    \\
    \label{eq:strongdomadapt}
      \|DW_t\mid_{E^c_x}\|\le e^{\lambda t}
    &<e^{\alpha_1 t}=\|DW_t\mid_{F_x}\|
    <e^{\eta t}\le\|DW_t\mid_{E^u_x}\|;
  \end{align}
  and also $K=1$ on items (1-3), by using the result
  from~\cite{Goum07}.

  In particular, note that the subbundle
  $F_0\oplus E^c\oplus F$ has a sectional-hyperbolic
  splitting.  Indeed, we have that $F_0$ is uniformly
  contracted, $E^c$ is two-dimensional and area expanding,
  and $F$ is uniformly expanding; and thus $E^c\oplus F$ is
  sectionally-expanded. In fact, since the bundles are
  orthogonal we can argue as in
  Subsection~\ref{sec:norm-adapted-singul}: for any given
  fixed $t>0$, a unit vector $f_1\in F$ and a unit vector
  $v\in E^c$ we have, for any unit vector $\tilde v\in E^c$
  such that $DX_{-t} \tilde v$ is othogonal to $DX_{-t}v$,
  the following bound
  \begin{align*}
    \|(\wedge^2DW_{-t})(f_1\wedge v)\|
    &=
    \|DW_{-t}f_1\|^2\|DW_{-t}v\|^2-(DW_{-t}f_1,DW_{-t}v)^2
    \\
    &=
    \|DW_{-t}f_1\|^2\|DW_{-t}v\|^2
    \\
    &\le
      e^{2(\lambda-\alpha_1)t}\|DW_{-t}\tilde v\|^2\|DW_{-t}v\|^2
      =
      e^{2(\lambda-\alpha_1)t}\|(\wedge^2DW_{-t})(\tilde
      v\wedge v)\|^2
    \\
    &\le
      e^{2(\lambda-\alpha_1)t}\|\wedge^2DX_{-t}\mid \wedge^2E^c\|^2
      \le
      e^{2(\lambda-\alpha_1)t}\cdot K^2e^{-2\lambda t}
      =
      K^2e^{-2\alpha_1 t};
  \end{align*}
  where we have used the domination
  relations~\eqref{eq:strongdom} in the following form
  \begin{align*}
    1=
    \frac{\|\tilde v\|}{\|f_1\|}
    =
    \frac{\|DW_t\cdot DW_{-t}\tilde v\|}{\|DW_t\cdot DW_{-t}
    f_1\|}
    <
    e^{(\lambda-\alpha_1) t}\frac{\|DW_{-t}\tilde v\|}{\|DW_{-t}f_1\|}
  \end{align*}
  to pass from $f_1$ to $\tilde v$ together with item (4).

  Hence, after Lemma~\ref{le:aux1}, we have
  $\|\wedge^2DW_{-t}\mid \wedge^2(E^c\oplus F)\|\le Ke^{-\alpha_1t}$,
  $t>0$. Moreover, analogously, given $t>0$ and  unit
  vectors $f_0,v$ on $F_0$ and $E^c$, respectively, choosing now a unit
  $\tilde v\in E^c_x$ such that $DW_t\tilde v$ is orthogonal to
  $DX_tv$, we obtain
  \begin{align*}
    \|(\wedge^2DW_{t})(f_0\wedge v)\|^2
    &=
      \|DW_tf_0\|^2\cdot\|DW_tv\|^2
    \le
      e^{-2\lambda t}\|DW_t\tilde v\|^2\cdot\|DW_tv\|^2
    \\
    &\le
      e^{-2\lambda t}\|\wedge^2DW_t\cdot \tilde
      v\wedge v\|^2
      \le
      e^{-2\lambda t}\|\wedge^2DW_t\mid_{\wedge^2 E^c}\|^2.
  \end{align*}
    This shows that the following is a
  dominated splitting of $G=\wedge^2(F_0\oplus E^c\oplus F)$
  with respect to the action of $\wedge^2DW_t\mid G$:
  \begin{align*}
    G=\wedge^2(F_0\oplus E^c)\oplus \big((F_0\oplus
    E^c)\wedge F\big)=G_0\oplus G_1;
  \end{align*}
  where $G_0=\wedge^2(F_0\oplus E^c)$ and
  $G_1=(F_0\oplus E^c)\wedge F$ is the subbundle generated
  by the exterior products $u\wedge f$ for any
  $u\in F_0\oplus E^c$ and $f\in F$. In addition, $G_0$ also
  admits a partially hyperbolic splitting
  $F_0\wedge E^c\oplus\wedge^2E^c=\wt{F_0}\oplus\wt{E^c}$
  just like in the proof of
  Theorem~\ref{mthm:singadaptmetric}

  Thus, by \cite{Goum07}, there exists an adapted norm
  $|\cdot|_*$ on $G$ for the action of $\wedge^2DW_t\mid G$:
  this norm restricted to $G_0$ satisfies the assumptions of
  Lemma~\ref{le:final} with $m=3$, $E=F_0$ and $F=E^c$; and
  there exists $\xi>0$ so that for all
  $u,v,w\in F_0\oplus E^c$ and $t>0$
  \begin{align}
    \label{eq:dominaquadrado}
    \frac{|\wedge^2DW_t \cdot (u\wedge v)|_*}{|\wedge^2DW_t\cdot
    (w\wedge f)|_*} < e^{-\xi t}\frac{|u\wedge
    v|_*}{|w\wedge f|_*}.
  \end{align}
  We observe that both
  $G_0=\wedge^2(F_0\oplus E^c)\simeq F_0\oplus E^c$ and
  $\wedge^2(E^c\oplus F)\simeq E^c\oplus F$ since the
  dimensions are the same. Hence, by Lemma~\ref{lemma189},
  there are Riemannian norms $|\cdot|_0$ on $F_0\oplus E^c$
  and $|\cdot|_1$ on $E^c\oplus F$ which induce $|\cdot|_*$
  restricted to $G_0$ and $\wedge^2(E^c\oplus F)$,
  respectively. We write $\langle\cdot,\cdot\rangle_i$ for the innner
  products generating $|\cdot|_i,i=1,2$.

  Applying Lemma~\ref{le:final} to $(G_0,|\cdot|_*)$ and
  $(F_0\oplus E^c,|\cdot|_0)$ we
  obtain~\eqref{eq:def-dom-split-cocycle} with $F_0$ in the
  place of $F$. Moreover, fixing $t>0$, if we define the
  Riemannian metric
  $|\alpha_0f_0+v+\alpha_1f|^2=\beta^2|\alpha_0f_0+v|_0^2+|\alpha_1f|_1^2$
  on $F_0\oplus E^c\oplus F$ for any given unit vector
  $v\in E^c$ and $\alpha_0,\alpha_1\in\RR$ and some
  $\beta>0$, we get
  \begin{align*}
    \frac{|DW_t v|}{|DW_t f|}
    &=
      \beta\frac{|DW_tv|_0}{|DW_tf|_1}
      =
      \beta\frac{|\wedge^2DW_t\cdot(v\wedge\tilde v)|_*}
      {|\wedge^2DW_t\cdot(f\wedge u)|_*}
      \le
      e^{-\xi t}\cdot\beta\frac{|v\wedge\tilde v|_*}{|f\wedge u|_*}
  \end{align*}
  where $\tilde v\in E^c$ and $u\in E^c\oplus F_1$ are such
  that $|DW_t\tilde v|_0=1$ and
  $\langle DW_t\tilde v,DW_tv\rangle_0=0$ and
  $\langle DW_tu,DW_tf\rangle_1=0$. Note that the
  introduction of $\beta$ does not change any of the relations
  obtained for $|\cdot|_0$.

  Now we argue similarly as in the proof of
  Lemma~\ref{le:final}: we choose $\beta>0$ so that
  \begin{align*}
    \beta\cdot
    \sup_{x\in\wh{\Gamma},0<t\le1}
    \left\{
    \frac{|v\wedge DX_{-t}\tilde v|_*}{|f\wedge DX_{-t}u|_*}:
    v\in E^c_x; \,\tilde v, u\in E^c_{X_tx}
    \right\}
    \le1.
  \end{align*}
  This immediately provides the following for $0<t\le1$
  \begin{align}\label{eq:domF1}
  |DX_t\mid E^c|=|DW_t\mid E^c|\le e^{-\xi
    t}\mm(DW_t|F)=e^{-\xi t}e^{\mu_1t}.
  \end{align}
  For general $t>0$ we again apply Lemma~\ref{le:subaddinduction} with
  $a=1$.

  Now we check that the Riemannian metric defined by
  \begin{align*}
    \|w+\alpha_0f_0+v+\alpha_1f+u\|_*^2
    =
    \|w\|^2+|\alpha_0f_0+v+\alpha_1f|^2+\|u\|^2
  \end{align*}
  for
  $w\in E^s_x, v\in E^c_x, u\in E^u_x,
  \alpha_0.\alpha_1\in\RR$ and $x\in\wh{\Gamma}$ is a
  sectional-adapted metric when restricted to
  $E^s\oplus E^c\oplus E^u$.

  We just have the use the previous estimates: we already have
  uniformly adapted contraction along $E^s$ and expansion along $E^u$
  from items (1) and (3). For the domination relations, for $t>0$ we
  estimate\footnote{We write $\mm_*$ for the conorm associated to the
    norm $\|\cdot\|_*$.}  using~\eqref{eq:strongdom}
  and~\eqref{eq:domF0} and $|\cdot|=|\cdot|_0$ norm on $F_0\oplus E^c$
  \begin{align*}
    \frac{\|DX_t\mid E^s\|_*}{\mm_*(DX_t\mid E^c)}
    &=
      \frac{\|DX_t\mid E^s\|}{\|DW_t\mid F_0\|}
      \cdot
      \frac{\|DW_t\mid F_0\|}{\mm_*(DW_t\mid E^c)}
    =
    e^{-\lambda t}
    \cdot
    \frac{|DW_t\mid F_0|}{\mm(DW_t\mid E^c)}
    \le
      e^{-\lambda t}e^{-\sigma t};
    \\
    \frac{\|DX_t\mid E^c\|_*}{\mm_*(DX_t\mid E^u)}
    &=
      \frac{|DW_t\mid E^c|}{e^{\alpha_1 t}}
      \cdot
      \frac{e^{\alpha_1 t}}{\mm(DW_t\mid E^u)}
      \le
      e^{-\xi t}
      \cdot
      \frac{\|DW_t\mid F\|}{\mm(DW_t\mid E^u)}
      \le
      e^{-\xi t}e^{-\eta t};
  \end{align*}
  where we have used $\|DW_t\mid F_0\|=|DW_t\mid F_0|$ since
  $\dim F_0=1$ and analogously for $F$.  For the sectional expansion,
  we already have for $t>0$
  \begin{align*}
    \|\wedge^2DX_t\mid \wedge^2E^c\|_*
    =
    |\wedge^2DW_t\mid \wedge^2E^c|
    =
    |\wedge^2DW_t\mid \wedge^2E^c|_0
    \ge
    e^{\sigma t}
  \end{align*}
  for some $\sigma>0$. Since $E^c, F$ and $E^u$ are orthogonal for the
  metric inducing $\|\cdot\|_*$, for any given $v\in E^c, u\in E^u$ of
  unit norm, and $\tilde v\in E^c$ so that $\|DX_{-t}\tilde v\|_*=1$
  and $\langle DX_{-t}\tilde v,DX_{-t}v\rangle_*=0$, we obtain using
  the domination
  \begin{align*}
    \|\wedge^2DX_{-t}\cdot(v\wedge u)\|_*
    &=
    \|DX_{-t} v\|_*\|DX_{-t} u\|_*
    \le
    e^{-(\xi+\eta)t}\|DX_{-t}v\|_*\|DX_{-t}\tilde v\|_*
    \\
    &=
    e^{-(\xi+\eta)t}\|\wedge^2DX_{-t}\cdot(v\wedge \tilde
    v)\|_*
    \le
    e^{-(\xi+\eta+\sigma)t}\|v\wedge \tilde v\|_*.
  \end{align*}
  Finally, for $u,v\in E^u$ we estimate just like
  in~\eqref{eq:secexpu} to obtain
  \begin{align*}
    \|(\wedge^2DX_{-t})(u\wedge v)\|_*
    =
    \|(\wedge^2DX_{-t})(u\wedge \tilde v)\|_*
    \le
    e^{-2\eta t}\|u\wedge v\|_*
  \end{align*}
  where $u\wedge v=u\wedge\tilde v$ so that
  $DX_{-t}\tilde v$ is orthogonal to $DX_{-t}u$.  This
  coupled with Lemma~\ref{le:aux1} is enough to conclude
  that
  $\|\wedge^2DX_{-t}\mid\wedge^2(E^c\oplus
  E^u)\|_*<e^{-2\min\{\eta,\sigma\}t}$ for each $t>0$.

This completes the proof of
Theorem~\ref{mthm:sectadaptmetricsing}.
\end{proof}


\def\cprime{$'$}


\end{document}